\documentclass[12pt]{amsart} %leqno is the option to put formula numbers on the left side
\usepackage{amsmath, amssymb, mathscinet}
\usepackage{amsthm} %theorem environment option
\usepackage[all]{xy}
\usepackage{enumerate}
\usepackage[T1]{fontenc}
\usepackage{hyperref} % for pdflatex
\usepackage{mathrsfs} % \mathscr
\usepackage{graphicx}
\renewcommand{\thefootnote}{} %footnote counter
\theoremstyle{plain} %text of this environment is typesetted in italics
\newtheorem{theorem}{\indent\sc Theorem}[section]
\newtheorem{lemma}[theorem]{\indent\sc Lemma}
\newtheorem{corollary}[theorem]{\indent\sc Corollary}
\newtheorem{proposition}[theorem]{\indent\sc Proposition}

\theoremstyle{definition} %text of this environment is typesetted in roman letters
\newtheorem{definition}[theorem]{\indent\sc Definition}
\newtheorem{remark}[theorem]{\indent\sc Remark}
\newtheorem{example}[theorem]{\indent\sc Example}

\newcommand{\C}{\mathbb{C}}
\newcommand{\R}{\mathbb{R}}

\newcommand{\Z}{\mathbb{Z}}
\newcommand{\N}{\mathbb{N}}

\def\dim{\mathop{\mathrm{dim}}\nolimits}

\newcommand{\abs}[1]{\left\lvert#1\right\rvert}

\newcommand{\norm}[1]{\lVert#1\rVert}

\newcommand{\calL}{\mathcal{L}}
\newcommand{\calO}{\mathcal{O}}

\newcommand{\calLAff}{\mathcal{L}_{\mathrm{Aff}}}
\newcommand{\OdX}{\calO \partial X}
\newcommand{\OdoX}{\calO \partial_O X}
\newcommand{\loge}{\log^\epsilon}
\newcommand{\expe}{\exp_\epsilon}
\newcommand{\qds}{\rho_\epsilon}
\newcommand{\radcon}{\phi}
\newcommand{\ds}[1]{\,\overline{#1}\,}
\newcommand{\Rp}{\R_{\geq 0}}
\newcommand{\famP}{\mathbb{P}} % for relatively hyperbolic groups
\newcommand{\classNPCPC}{\mathcal{C}}
\newcommand{\LieG}{\mathbf G}
\usepackage{
  pst-poly,
  pstricks-add
}

\makeatletter
\def\address#1#2{\begingroup
\noindent\parbox[t]{7.8cm}{%
\small{\scshape\ignorespaces#1}\par\vskip1ex
\noindent\small{\itshape E-mail address}%
\/: #2\par\vskip4ex}\hfill%
\endgroup}%
\makeatother
\title[coarse Cartan-Hadamard theorem]{\uppercase{A coarse Cartan-Hadamard theorem with application to the coarse Baum-Connes conjecture}}
\author{
\textsc{Tomohiro Fukaya, Shin-ichi Oguni} %names of authors
}
\date{} %leave empty
\begin{document}

%\maketitle

%%%%%%%%%%%%%%% footnote %%%%%%%%%%%%%%%%
 %2010 MSC numbers
\footnote{ %2010 MSC numbers
}
\footnote{ %key words and phrases
}

%\footnote{ %acknowledgment of support etc. if any
%T.Fukaya S.Oguni were supported by Grant-in-Aid for Young Scientists (B)
%(15K17528)  and (16K17595), respectively, 
%from Japan Society of Promotion of Science
%}
% or Using 
\thanks{T.Fukaya S.Oguni were supported by Grant-in-Aid for Young Scientists (B)
(15K17528)  and (16K17595), respectively, 
from Japan Society of Promotion of Science}
\renewcommand{\thefootnote}{\fnsymbol{footnote}} %footnote counter
%%%%%%%%%%%%%%%%%%%%%%%%%%%%%%%%%%%%%%%%

\begin{abstract}
We establish a coarse version of the Cartan-Hadamard theorem, 
%for %proper geodesic coarsely convex spaces. 
which states that proper
coarsely convex spaces are coarsely homotopy equivalent to
the open cones of their ideal boundaries.
As an application,
we show that such spaces satisfy the coarse Baum-Connes
conjecture. Combined with the result of Osajda-Przytycki,
it implies that systolic groups and locally finite systolic
complexes satisfy the coarse Baum-Connes conjecture.
\end{abstract}

\maketitle

\section{Introduction} %delete * to number this section
\label{sec:introduction}
%Acknowledgments may be included at the end of the Introduction.

%Let $X$ be a metric space.  A {\itshape path} in $X$ is a map
%$\gamma\colon I\to X$, where $I=[0,a]$ or $I=\R_{\geq 0}$.

%Let $\calL$ be a family of paths in
%$X$. The family $\calL$ is {\itshape symmetric} if a path
%$\gamma$ is in $\calL$, then so is its inverse path $\gamma^{-1}$,
%%
%and, $\calL$ is {\itshape prefix closed} if a path 
%$\gamma\colon [a,b]\to X$ is in $\calL$ then so is the 
%restriction of $\gamma$ to $[a,c]$, for all $c\in [a,b]$.
%See Section~\ref{sec:coars-conv-spac} for detail.

%\begin{definition}
%\label{def:ge-cBNC}
 % We choose a base point vertex $O\in  X^{(0)}$.  

The metric on a geodesic space $X$ is said to be {\itshape
convex} if all geodesic segments $\gamma_1\colon [0,a_1]\to X$ and
$\gamma_2\colon [0,a_2]\to X$ satisfy the inequality
\begin{align*}
 \ds{\gamma_1(ta_1),\gamma_2(ta_2)}\leq (1-t)\ds{\gamma_1(0),\gamma_2(0)}
 +t\ds{\gamma_1(a_1),\gamma_2(a_2)},
\end{align*}
for all $t\in [0,1]$, where we denote by $\ds{x_1,x_2}$ the distance
between $x_1$ and $x_2$.  This condition generalizes metric properties
of simply connected complete Riemannian manifolds with non-positive
sectional curvature. A geodesic space with a convex metric is also
called a Busemann non-positively curved space. Unlike Gromov's
definition of hyperbolicity of metric spaces, convexity does not behave
well under coarse equivalences of geodesic spaces even if we
allow bounded errors in the inequality. Indeed, the 2-dimensional vector
space $\R^2$ with the $l_1$-metric contains {\itshape fat 2-gons}, and
so the $l_1$-metric is not convex, although it is coarsely
equivalent to the $l_2$-metric, which is convex.
An idea to overcome this problem is to consider a particular
subfamily of geodesics.

Let $X$ be a  metric space. 
Let $C\geq 0$ be a constant.
Let $\calL$ be a family of geodesic segments.
The space $X$ is 
{\itshape geodesic $(C,\calL)$-coarsely convex}, 
if $C$ and $\calL$ satisfy the following.
\begin{enumerate}[(i)] 
 \item \label{ge-conn}       
       For $v,w\in X$, there exists a 
       geodesic segment
       $\gamma\in \calL$ 
       with $\gamma(0) = v$ and $\gamma(\ds{v,w}) = w$.
%       Here we denote by $\ds{v,w}$ the distance between $v$ and $w$.
 \item \label{ge-qconvex}
       Let $\gamma, \eta \in \calL$ be geodesic segments
       such that $\gamma\colon [0,a]\to X$, $\eta\colon [0,b]\to X$. 
       For $t\in [0,a]$, $s\in [0,b]$ and
       for $0\leq c\leq 1$, we have that        
       \begin{align*}
	\ds{\gamma(ct),\eta(cs)}
	\leq c\ds{\gamma(t),\eta(s)} + (1-c)\ds{\gamma(0),\eta(0)}+ C.
	%	, where $C$ is a uiversal constant.
       \end{align*}
\end{enumerate} 
The family $\calL$ satisfying (\ref{ge-conn}) and (\ref{ge-qconvex})
is called a {\itshape system of good geodesics}, and elements 
 $\gamma\in \calL$ are called {\itshape good geodesics}.
%\end{definition}

We say that a metric space $X$ is a 
{\itshape geodesic coarsely convex} space if
there exist a constant $C$ and a family of
geodesics $\calL$ such that $X$ is geodesic 
$(C,\calL)$-coarsely convex. 

Being geodesic coarsely convex is not invariant under coarse equivalence
yet. In Section~\ref{sec:coars-conv-spac}, we introduce an
alternative definition, we say {\itshape coarsely convex}, using
quasi-geodesics, and show that it is invariant under coarse equivalence.
We remark that geodesic coarsely convex spaces are coarsely convex
spaces.  For a coarsely convex space $X$, the {\itshape ideal boundary},
denoted by $\partial X$, is a set of equivalence classes of
quasi-geodesic rays which can be approximated by elements of $\calL$,
equipped with a metric given by the ``Gromov product''.

Suppose that $N$ is a connected, simply connected, complete,
Riemannian $n$-manifold with all sectional curvatures being less than or
equal to zero. It follows from the Cartan-Hadamard theorem that $N$ is
diffeomorphic to the Euclidean space $\R^n$. We remark that the ideal
boundary of $N$ is homeomorphic to the $(n-1)$-sphere $S^{n-1}$, and
$\R^n$ is regarded as the {\itshape open cone} over $S^{n-1}$.  The main
result of this paper is a coarse geometric analogue of this theorem.

\begin{theorem}
\label{thm:cCATHAD}
 Let $X$ be a proper coarsely convex
 space. Then $X$ is coarsely homotopy equivalent to $\OdX$,
 the open cone over the ideal boundary of $X$.
\end{theorem}

The class of geodesic coarsely convex spaces includes
geodesic Gromov hyperbolic spaces~\cite[\S 2, Proposition 25]{MR1086648}
and CAT(0)-spaces, more generally, 
Busemann non-positively curved spaces~\cite{MR1744486}\cite{MR2132506}.
We remark that this class is closed under direct product, therefore, 
it includes products of these spaces.
%geodesic Gromov hyperbolic spaces and CAT(0)-spaces.
An important subclass of geodesic coarsely convex spaces is 
a class of {\itshape systolic complexes}.

Systolic complexes are connected, simply connected simplicial complexes
with combinatorial conditions on links. They satisfy one of the basic
feature of CAT(0)-spaces, that is, the balls around convex sets are
convex. This class of simplicial complexes was introduced by
Chepoi~\cite{Chepoi-graph-CAT0} (under the name of {\itshape bridged
complexes}), and independently, by 
Januszkiewich-\'Swi\polhk{a}tkowski~\cite{Janus-Switk-smpl-NPC} and
Haglund~\cite{Haglund-cmplx-smp-hyp-grnd-dim}.
Osajda-Przytycki~\cite{bdrySytolic} introduced {\itshape Euclidean
geodesics}, which behave like CAT(0) geodesics, to construct boundaries
of systolic complexes. Their result implies the following.

%Gromov\cite{MR919829} gave a combinatorial characterization of
%CAT(0)-cubical complexes. Based on this result,
%Chepoi~\cite{Chepoi-graph-CAT0} characterize Bridged complexes, which
%satisfy one of the basic features of CAT(0)-spaces, that is, the balls
%around convex sets are convex.
%Januszkiewich-\'{S}wi\k{a}tkowski~\cite{Janus-Switk-smpl-NPC} and
%Haglund~\cite{Haglund-cmplx-smp-hyp-grnd-dim} re-discovered this class
%of simplicial complex, and named them systolic complexes, and studied
%from the view point of the geometric group theory. To construct boundaries
%of systolic complexes, Osajda-Przytycki~\cite{bdrySytolic} introduced
%{\itshape Euclidean geodesics}, which behave like CAT(0) geodesics.

%They make use of minmal surface theory developed by
%Elsner\cite{Elsner-flats-systolic}.

\begin{theorem}
[{\cite[Corollary 3.3, 3.4]{bdrySytolic}}]
%[Osajda-Przytycki \cite{bdrySytolic}]
\label{thm:systolic-cConvex}
 The 1-skeleton of systolic complexes are geodesic coarsely convex spaces.
\end{theorem}
A group is {\itshape systolic} if it acts geometrically by simplicial
automorphisms on a systolic complex. Osajda-Przytycki used their result
to show that systolic groups admit $EZ$-structures. This implies the
Novikov conjecture for torsion-free systolic groups.  Now it is
natural to ask whether systolic groups satisfy the {\itshape the coarse
Baum-Connes conjecture}.
%\subsection{Coarse Baum-Connes conjecture}

Let $X$ be a proper metric space. The {\itshape coarse assembly map} is
a homomorphism from the {\itshape coarse K-homology} of $X$ to the
K-theory of the {\itshape Roe-algebra} of $X$.  The coarse Baum-Connes
conjecture~\cite{MR1388312} states that for ``nice'' proper metric
spaces, the coarse assembly maps are isomorphisms.

As a corollary of Theorem~\ref{thm:cCATHAD}, we have the following.
\begin{theorem}
\label{thm:cconvex-cBC}
 Let $X$ be a proper coarsely convex space.
 Then $X$ satisfies the coarse Baum-Connes conjecture.
\end{theorem}

The coarse Baum-Connes conjecture is known to be true for several classes of 
proper metric spaces. Examples of such classes are following.
\begin{enumerate}
 \item \label{item:Ghyp} Geodesic Gromov hyperbolic spaces
       \cite{MR1388312}\cite{WillettThesis}.
 \item \label{item:BNPC}
       Busemann non-positively curved spaces
       \cite{MR1388312}\cite{WillettThesis}\cite{Busemann_cBC}.
 \item \label{item:prod}
       Direct products of geodesic Gromov hyperbolic spaces and Busemann
       non-positively curved spaces \cite{product-cBC}.
 \item \label{item:embHilb}
       Metric spaces which admit coarse embeddings into the Hilbert space
       \cite{MR1728880}.
\end{enumerate}

Theorem~\ref{thm:cconvex-cBC} covers examples
(\ref{item:Ghyp}),~(\ref{item:BNPC}) and~(\ref{item:prod}) in the above
list.  Combining it with Theorem~\ref{thm:systolic-cConvex}, we obtain
the following.
\begin{corollary}
\label{cor:syst-cBC}
 Let $X$ be a locally finite systolic complex. Then $X$ satisfies the 
 coarse Baum-Connes conjecture. Especially, systolic groups satisfy 
 the coarse Baum-Connes conjecture.
\end{corollary}

Recently, Osajda-Huang~\cite{LArtin-systolic} showed that 
Artin groups of almost large-type are systolic groups, and
Osajda-Prytu{\l}a~\cite{Classifying-sp-systolic} showed that graphical small
cancellation groups are systolic groups. 
%there exists a graphical small cancellation group such that a family of expander graphs is embedded isometrically into its Cayley graph.
We remark that large-type Artin groups are of almost large-type,
and it is unknown whether these groups act geometrically on
CAT(0)-spaces.

\begin{corollary}
 Artin groups of almost large type and graphical small cancellation
 groups satisfy the coarse Baum-Connes conjecture. 
\end{corollary}

Corollary~\ref{cor:syst-cBC} with a descent principle implies
the Novikov conjecture for systolic groups.  As already mentioned, it is
known to be true~\cite{bdrySytolic}.  In fact we can show the
Novikov conjecture for wider classes of groups since the coarse
Baum-Connes conjecture is stable under taking product with any
polycyclic group and is studied well for relatively hyperbolic groups.
Let $\classNPCPC$ be a class of groups consisting of direct products
of hyperbolic groups, CAT(0)-groups, systolic groups, and polycyclic
groups. Note that a polycyclic group with a word metric is not
necessarily coarsely convex. 
We give details in Remark 6.11.
\begin{theorem}
 \label{thm:relhypPC} Let a finitely generated group $G$ be 
 one of the following: 
 \begin{enumerate}%[\hspace{1.2em}$(1)$]
  \item a member of $\classNPCPC$, 
  \item\label{second} a group which is hyperbolic relative to a
       finite family of
       subgroups belonging to $\classNPCPC$,
  \item a group which is the direct product of a group as in (\ref{second})
	and a polycyclic group.
 \end{enumerate}
 Then the group $G$
 satisfies the coarse Baum-Connes conjecture. Moreover, if 
 $G$ is torsion free, then $G$ 
satisfies the Novikov conjecture.
\end{theorem}
To the best knowledge of the authors, it is unknown whether
each group $G$ in Theorem~\ref{thm:relhypPC} admits
an $EZ$-structure or not.

%\begin{theorem}
%\label{thm:free-prod}
% Let $G_1,\dots,G_n$ be groups in $\classNPCPC$, and let $H$ be a polycyclic 
% groups. Then the group
% \begin{align*}
%  G:= (G_1*\dots*G_n)\times H
% \end{align*}
% satisfies the coarse Baum-Connes conjecture. Here 
% $G_1*\dots*G_n$ is the free product of $G_1,\dots, G_n$. 
% Moreover, if $G$ is torsion-free, then $G$ satisfies the Novikov conjecture.
%\end{theorem}

Finally, we mention some algebraic properties of groups acting
geometrically on coarsely convex spaces.  These are direct
consequences of semihyperbolicity of coarsely convex spaces and
results of Alonso and Bridson~\cite{A-BSemihyp}.

\begin{corollary}
 \label{cor:CCGroup}
 Let $G$ be a group acting on a coarsely convex spaces $X$ properly and
 cocompactly by isometries. 
Then the following hold.
 \begin{enumerate}
%  \item \label{item:FP} $G$ is finitely presented.
  \item \label{item:FPinfty} $G$ is finitely presented and 
	of type $FP_{\infty}$.
  \item \label{item:isoperim} $G$ satisfies a quadratic isoperimetric inequality.
 \end{enumerate}
 Moreover, suppose that a system of good quasi-geodesic
 segments $\calL$ of $X$ is $G$-invariant, then
 \begin{enumerate}
 \setcounter{enumi}{2}
  \item \label{item:conjugacy-prob} $G$ has a solvable conjugacy problem.
  \item \label{item:plycsubgrp} Every polycyclic subgroup of $G$ contains
	a finitely generated abelian subgroup of finite index.
 \end{enumerate}
\end{corollary}

\begin{remark}
 It is already known that systolic groups satisfy all properties
 mentioned in Corollary~\ref{cor:CCGroup}, since
 Januszkiewich-\'Swi\polhk{a}tkowski~\cite{Janus-Switk-smpl-NPC} proved
 that systolic groups are biautomatic.
\end{remark}

The organization of the paper is as follows. 
In Section~\ref{sec:coarse-geometry}, we briefly review coarse geometry,
and give the definition of coarse homotopy.
In Section~\ref{sec:coars-conv-spac}, we introduce coarsely convex
spaces, and we show that it is invariant under coarse equivalence.
In Section~\ref{sec:ideal-boundary}, we construct the ideal boundary,
then we introduce the {\itshape Gromov product} to define a topology on
the boundary. In Section~\ref{sec:main-result}, we give a proof of
Theorem~\ref{thm:cCATHAD}.  
In Section~\ref{sec:appl-coarse-baum}, we discuss on the relation with
the coarse Baum-Connes conjecture. We give a proof of
Theorem~\ref{thm:cconvex-cBC}. We also show that the coarse $K$-homology
of a coarsely convex space is isomorphic to the reduced $K$-homology of
its ideal boundary. Then we discuss on the direct product with
polycyclic groups, and on relatively hyperbolic groups.
In Section~\ref{sec:groups-acting-coars}, we show that a coarsely convex
space is semihyperbolic in the sense of Alonso-Bridson, and we mention
that Corollary~\ref{cor:CCGroup} follows from this fact.
In Section~\ref{sec:funct-analyt-char}, we give a functional analytic
characterization of the ideal boundary. As a corollary, we obtain that
the ideal boundary coincides with the combing corona in the sense of 
Engel and Wulff~\cite{combable-corona}.

\subsection*{Acknowledgments}
We are grateful to Damian Osajda for suggesting us a problem on the
coarse Baum-Connes conjecture for systolic complexes, which led us to
the definition of the coarsely convex spaces. We also appreciate his many
useful comments on the first draft of this paper. We thank Takumi Yokota
for helpful discussions.

%The first author and the second author were supported by Grant-in-Aid
%for Young Scientists (B) (15K17528) and (16K17595), respectively, from
%Japan Society of Promotion of Science.

\section{Coarse geometry}
\label{sec:coarse-geometry}
In this section we briefly review coarse geometry.
For points $v,w\in X$, we denote by $\ds{v,w}$ the distance between
$v$ and $w$. For $r\geq 0$ and for a subset $K\subset X$, we denote by 
$B_r(K)$ the closed $r$-neighbourhood of $K$ in $X$.

\subsection{Coarse map}
Let $X,Y$ be metric spaces. Let $f\colon X\to Y$ be a map.
\begin{enumerate}
 \item 
       The map $f$ is {\itshape bornologous} 
       if there exists a non-decreasing function
       $\theta\colon \R_{\geq 0} \to \R_{\geq 0}$ such that 
       for all $x,x'\in X$, we have
       \[
       \ds{f(x),f(x')} \leq \theta(\ds{x,x'}).
       \]       
  \item The map $f$ is {\itshape proper} if for each bounded subset
	$B\subset Y$, the inverse image $f^{-1}(B)$ is bounded.
  \item The map $f$ is {\itshape coarse} if it is bornologous and proper.
\end{enumerate}
For maps $f,g\colon X\to Y$, we say that $f$ and $g$ are {\itshape
close} if there exists a constant $C\geq 0$ such that
$\ds{f(x),g(x)}\leq C$ for all $x\in X$. A coarse map $f\colon X\to Y$
is a {\itshape coarse equivalence map} 
if there exists a coarse map $g\colon Y\to X$
such that the composites $g\circ f$ and $f\circ g$ are close to the identity
$\mathrm{id}_X$ and $\mathrm{id}_Y$, respectively.
We say that $X$ and $Y$ are {\itshape coarsely equivalent} if
there exists a coarse equivalence map $f\colon X\to Y$.

%Higson-Roe~\cite{MR1388312} introduced a
% notion of coarse homotopy. After that, they gave an
% alternative definition of coarse homotopy, which is a variant of Lipschitz
%homotopy. (For Lipschitz homotopy, see \cite[1.$C_3$]{MR919829},
%\cite[Definition 4.1]{MR1344138} and \cite[Definition 11.1]{MR1451755}.)
%Our definition is based on \cite[Section 11]{MR1817560} 
%and \cite[Definition 3.9]{WillettThesis}.

There exists a weaker equivalence relation between coarse maps, which plays
an important role for an algebraic topological approach to the coarse
Baum-Connes conjecture.
\begin{definition}
\label{def:chomotopy}
 Let $f,g\colon X\to Y$ be coarse maps between metric
 spaces. The maps $f$ and $g$ are 
 {\itshape coarsely homotopic} if there exists a metric subspace 
$Z = \{(x,t):0\leq t\leq T_x\}$ 
of $X\times \R_{\geq 0}$ and a coarse map
 $h\colon Z\to Y$, such that
\begin{enumerate}
 \item the map $X \ni x\mapsto T_x\in \R_{\geq 0}$ is bornologous,
 \item $h(x,0) = f(x)$, and
 \item $h(x,T_x) = g(x)$.
\end{enumerate}
Here we equip $X\times \R_{\geq 0}$ with the $l_1$-metric, that is, 
$d_{X\times \R_{\geq 0}}((x,t),(y,s)):= \ds{x,y}+ \abs{t-s}$ for 
$(x,t),(y,s)\in X\times \R_{\geq 0}$.
\end{definition}
Coarse homotopy is then an equivalence relation on coarse maps.  A
coarse map $f\colon X\to Y$ is a {\itshape coarse homotopy equivalence
map} if there exists a coarse map $g\colon Y\to X$ such that
the composites $g\circ f$ and $f\circ g$ are coarsely homotopic to the
identity $\mathrm{id}_X$ and $\mathrm{id}_Y$, respectively.  We say that
$X$ and $Y$ are {\itshape coarsely homotopy equivalent} if there exists
a coarse homotopy equivalence map $f\colon X\to Y$.

\subsection{Quasi-isometry}
\label{sec:quasi-isometry} Let $\lambda\geq 1$ and $k\geq 0$ be
constants. Let $X$ and $Y$ be metric spaces.  We say that a map 
$f\colon X\to Y$ 
is a {\itshape $(\lambda,k)$-quasi-isometric embedding} if for
all $x,x'\in X$, we have
\begin{align*}
 \frac{1}{\lambda}\ds{x,x'} -k \leq 
 \ds{f(x),f(x')} \leq \lambda\ds{x,x'} + k.
\end{align*}

 Let $X'\subset X$ be a subset. For $M\geq 0$, we say that $X'$ is
 {\itshape $M$-dense} in $X$ if $X=B_M(X')$. 
 We say that a map $f\colon X\to Y$ is a
 {\itshape quasi-isometry} if there exist constants $\lambda,k,M$ such
 that $f$ is a $(\lambda,k)$-quasi-isometric embedding and the image $f(X)$ is
 $M$-dense in $Y$. We say that $X$ and $Y$ are {\itshape
 quasi-isometric} if there exists a quasi-isometry $f\colon X\to Y$.

% Let $I\subset R_{\geq 0}$ be a closed connected subset.  A
% $(\lambda,k)$-quasi-isometric embedding $\gamma\colon I\to X$ is called
% a {\itshape $(\lambda,k)$-quasi-geodesic}.  If $I=\Rp$, then
% $\gamma$ is called a $(\lambda,k)$-quasi-geodesic ray, and if
% $I=[a,b]$, then $\gamma$ is called a $(\lambda,k)$-quasi-geodesic
% segment.

%%
 A {\itshape $(\lambda,k)$-quasi-geodesic} in $X$ is a
 $(\lambda,k)$-quasi-isometric embedding $\gamma\colon I\to X$, 
 where $I$ is a closed connected subset of $\R$.  
 If $I=\Rp$, then we say that $\gamma$ is a {\itshape
 $(\lambda,k)$-quasi-geodesic ray}, and if $I=[0,a]$, then 
 we say that $\gamma$ is a {\itshape $(\lambda,k)$-quasi-geodesic segment}.

%\begin{definition}
%Let $\lambda\geq 1$ and $k\geq 0$ be constants. 

A metric space $X$ is {\itshape $(\lambda,k)$-quasi-geodesic} 
if for all $x,y\in X$, there exists 
a $(\lambda, k)$-quasi-geodesic segment $\gamma\colon [0,a]\to X$ with
$\gamma(0) = x$ and $\gamma(a) = y$.
%\end{definition}
%
We say that a metric space $X$ is {\itshape quasi-geodesic}
if there exist constants $\lambda$ and $k$ such that
$X$ is $(\lambda,k)$-quasi-geodesic.
The following criterion is well-known.
\begin{lemma}
 Let $X$ and $Y$ be quasi-geodesic spaces. Then $X$ and $Y$ are 
 coarsely equivalent if and only if $X$ and $Y$ are quasi-isometric.
\end{lemma}

%\begin{lemma}
%\label{lem:qg-born} 
%%
% Let $X$ be a $(\lambda,k)$-quasi-geodesic space, and let $Y$ be a
% metric space. A map $f\colon X \to Y$ is bornologous if there exists a
% constant $C>0$ such that for all $v,w\in X$ with 
 %%
% $\ds{v,w}\leq \lambda +k$, we have $\ds{f(v),f(w)}\leq C$.
%\end{lemma}

%\begin{proof}
% Let $C>0$ be a constant in the statement.
% Let $x,y\in X$ be points in $X$.
% There exists a $(\lambda,k)$-quasi-geodesic 
% $\gamma\colon [0,a]\to X$ such that $\gamma(0)=x$ and $\gamma(a) =y$.
% We remark that $a\leq \lambda\ds{v,w}+k$.

% Set $n:=\lceil a \rceil$. 
% We define a sequence $v_i$ by, $v_0 := v$, $v_i:=\gamma(i)$ 
% for $1\leq i\leq n-1$, and $v_n:=w$. 
%Then $\ds{v_{i-1},v_i}\leq \lambda +k$ for all $1\leq i\leq n$.
% It follows that 
% \begin{align*}
%  \ds{f(x),f(y)}\leq \sum_{i=1}^n\ds{f(v_{i-1}),f(v_i)}\leq Cn
%  \leq C(a+1)\leq C(\lambda\ds{v,w}+k+1).
% \end{align*}
%\end{proof}

%\begin{definition}
% Let $X$ be a metric space and let $X'\subset X$ be a subset. For 
% $M\geq 0$, we say that $X'$ is $M$-dense in $X$ if $X=B_M(X')$.
%\end{definition}

\subsection{Open cone}
Let $M$ be a compact metrizable space. 
The {\itshape open cone} over $M$, denoted by $\calO M$, is
the quotient $\Rp\times M/(\{0\}\times M)$. 
For $(t,x)\in \Rp\times M$,
we denote by $tx$ the point in $\calO M$ represented by $(t,x)$.

Let $d_M$ be a metric on $M$. % which is compatible with the topology.
We assume that the diameter of $M$ is at most 2.
We define a metric $d_{\calO M}$ on $\calO M$ by 
\begin{align*}
 d_{\calO M}(tx,sy):=\abs{t-s} + \min\{t,s\}d_M(x,y).
\end{align*}
We call $d_{\calO M}$ the induced metric by $d_M$. 

\begin{remark}
 When we take another metric $d'_M$ on $M$ such that the diameter of $M$
 is at most 2, we have the induced metric $d'_{\calO M}$ on $\calO M$ by
 $d'_M$.  Then the identity map $id_{\calO M}$ between $(\calO
 M,d_{\calO M})$ and $(\calO M,d'_{\calO M})$ is not necessarily a
 coarse equivalence map, in fact, it is not necessarily a coarse
 homotopy equivalence map.  
 Nevertheless a radial contraction gives a coarse homotopy equivalence map.
 We refer to \cite{MR1388312} and \cite{WillettThesis}.
%\begin{enumerate}
%\item 
%When we take another metric $d'_M$ on $M$ such that the diameter
%      of $M$ is at most 2,
%      we have the induced metric $d'_{\calO M}$ on $\calO M$ by $d'_M$.
%      Then the identity map $id_{\calO M}$ between $(\calO M,d_{\calO M})$ and
%      $(\calO M,d'_{\calO M})$
%      is not necessarily a coarse equivalence map, but it is known to be a coarse
%      homotopy equivalence map \cite{MR1388312}.
%\item Let $(H,d^H)$ be a Hilbert space.  Take a topological embedding of
%      $M$ into the unit sphere in a Hilbert space $H$. We denote by
%      $\iota:M\to H$ the embedding and by $d^{H}_M$ the metric on $M$
%      for that the embedding $\iota$ is isometric.  Then we have the
%      induced metric $d^H_{\calO M}$ by $d^H_M$.  Consider the extended
%      embedding $\iota:(\calO M,d^H_M)\to (H,d^H); tm\mapsto t\iota(m)$,
%      which is a coarse embedding \cite{WillettThesis}. Thus $(\calO M,
%      d^H_{\calO M})$
%      and $(\iota(\calO M),d^H|_{\iota(\calO M)\times\iota(\calO M)})$
%      are naturally coarsely equivalent.  
%\end{enumerate}
\end{remark}

\section{coarsely convex space}
\label{sec:coars-conv-spac}

\begin{definition}
\label{def:cBNC}
 % We choose a base point vertex $O\in  X^{(0)}$.  
Let $X$ be a metric space. Let $\lambda\geq 1$, $k \geq 0$, $E\geq 1$,
and $C\geq 0$ be constants. 
Let $\theta\colon \Rp\to\Rp$ be a non-decreasing function.
Let $\calL$ be a family of $(\lambda,k)$-quasi-geodesic segments.  
The metric space $X$ is {\itshape
$(\lambda,k,E,C,\theta,\calL)$-coarsely convex}, if
$\calL$ satisfies the following.
\begin{enumerate}[(i)$^q$] 
 \item \label{qconn}
       For $v,w\in X$, there exists a 
       quasi-geodesic segment
       $\gamma\in \calL$ with $\gamma\colon [0,a]\to X$,
       $\gamma(0) = v$ and $\gamma(a) = w$.
 \item \label{qconvex}
%       Set $I=[0,a]$ %or $I=\Rp$, 
%       and $J=[0,b]$. %or $J=\Rp$.
       Let $\gamma, \eta \in \calL$ be 
       quasi-geodesic segments 
       with $\gamma\colon [0,a]\to X$ and $\eta\colon [0,b]\to X$. 
       Then for 
       $t\in [0,a]$, $s\in [0,b]$, and $0\leq c\leq 1$, we have that 
       \begin{align*}
	\ds{\gamma(ct),\eta(cs)}
	\leq cE\ds{\gamma(t),\eta(s)} 
	      + (1-c)E\ds{\gamma(0), \eta(0)}+ C.
	%	, where $C$ is a uiversal constant.
       \end{align*}
 \item \label{qparam-reg}       
       Let $\gamma,\eta\in \calL$ be quasi-geodesic segments 
       with $\gamma\colon [0,a]\to X$ and $\eta\colon [0,b]\to X$.
       Then for $t\in [0,a]$ and $s\in [0,b]$, we have
       \begin{align*}
	\abs{t-s} \leq \theta(\ds{\gamma(0),\eta(0)}+\ds{\gamma(t),\eta(s)}).
       \end{align*}
%       For any parameter pair $(T,W)$ for $\calL$, 
%       the map $T\colon W\to \Rp$ is $\theta$-bornologous.
%       Here we equip $X\times X$ with the $l_1$-product metric.
       %and equip $W$ with its restriction.
%%
%       For all choice $O\in X$, $x\in X$, 
%       $\gamma_{Ox}\in \calL$ with 
%       $\gamma_{Ox}\colon [0,T_{Ox}]\to X$, 
%       $\gamma_{Ox}(0)=O$, and
%       $\gamma_{Ox}(T_{Ox})=x$, 
%       such that
%       $x\in \gamma_{Ox}$, 
%       there exists $T_{Ox}\in \Rp$ with 
%       the map $x\mapsto T_{Ox}$ is a $\theta$-bornologous map from $X$ to 
%       $\Rp$.
\end{enumerate} 
The family $\calL$ satisfying (\ref{qconn})$^q$, 
(\ref{qconvex})$^q$, and  (\ref{qparam-reg})$^q$
is called a {\itshape system of good quasi-geodesic segments}, and elements 
 $\gamma\in \calL$ are called {\itshape good quasi-geodesic segments}.
\end{definition}

We say that a metric space $X$ is a {\itshape coarsely convex} space if
there exist constants $\lambda,k,E,C$,
a non-decreasing function $\theta\colon \Rp\to \Rp$, 
and a family of
$(\lambda,k)$-quasi-geodesic segments $\calL$ such that $X$ is
$(\lambda,k,E,C,\theta,\calL)$-coarsely convex.  

We remark that if $\calL$ consists of only geodesic segments,
then $\calL$ satisfies~(\ref{qparam-reg})$^q$ by the triangle
inequality. Therefore geodesic $(C,\calL)$-coarsely convex spaces are
$(1,0,1,C,\mathrm{id}_{\Rp},\calL)$-coarsely convex. We also remark that
Gromov~\cite[6.B]{asym_invs} mentioned the inequality in
(\ref{qconvex})$^q$.
% and (\ref{HdistD}).

\begin{proposition}
\label{prop:ceq-pres-qconvity}
 Let $X$ and $Y$ be quasi-geodesic spaces such that $X$ and $Y$ are
 coarsely equivalent. If $X$ is coarsely convex, then so is $Y$.
\end{proposition}

\begin{proof}
 Let $X$ and $Y$ be quasi-geodesic spaces such that $X$ and $Y$ are
 coarsely equivalent. 
 There exist a map $f\colon X\to Y$ and $A\geq 1$
 such that $f(X)$ is $A$-dense in $Y$, and for all $x,x'\in X$,
 \begin{align*}
  \frac{1}{A}\ds{x,x'} -A \leq \ds{f(x),f(x')}\leq A\ds{x,x'} +A.
 \end{align*} 
Suppose that $X$ is 
$(\lambda,k,E,C,\theta,\calL_X)$-coarsely convex.
For points $p,q\in Y$ and a path $\gamma\colon[0,a]\to X$, we define 
a path $\gamma_{p,q}\colon [0,a]\to Y$ by 
\begin{align*}
 &\gamma_{p,q}(0):=p,  &\gamma_{p,q}(a):=q,& 
 &\gamma_{p,q}(t):= f\circ \gamma(t) \;\text{ for } t\in (0,a).
\end{align*}
If $\ds{p,f\circ \gamma(0)}\leq A$, $\ds{q,f\circ \gamma(a)}\leq A$ and
$\gamma$ is a $(\lambda,k)$-quasi-geodesic segment, then $\gamma_{p,q}$
is a $(A\lambda,A(k+3))$-quasi-geodesic segment in $Y$. Thus we define a
family of $(A\lambda,A(k+3))$-quasi-geodesic segments in $Y$, denoted by
$\calL_Y$, as a family consisting of all quasi-geodesic segments
$\gamma_{p,q}$ where $p,q$ are points in $Y$, and $\gamma$ is a
quasi-geodesic segment in $\calL_X$ such that $\ds{p,f\circ
\gamma(0)}\leq A$ and $\ds{q,f\circ \gamma(a)}\leq A$.

We will show that $\calL_Y$ satisfies the conditions 
in Definition~\ref{def:cBNC}.
It is clear that (\ref{qconn})$^q$ holds.
We consider (\ref{qconvex})$^q$. %Set $I=[0,a]$ and $J=[0,b]$.
 Let $\gamma, \eta \in \calL_Y$ be quasi-geodesic segments 
 such that $\gamma\colon [0,a]\to Y$, and $\eta\colon [0,b]\to Y$.
 Then there exist $\gamma', \eta'\in \calL_X$ such that
\begin{align*}
 &\ds{\gamma(0),f\circ \gamma'(0)}\leq A,
 &\ds{\gamma(a),f\circ \gamma'(a)}\leq A,&
 &\gamma(t)=f\circ \gamma'(t) \text{ for } t\in (0,a), \\
 &\ds{\eta(0),f\circ \eta'(0)}\leq A,
 &\ds{\eta(b),f\circ \eta'(b)}\leq A,&
 &\eta(s)=f\circ \eta'(s) \text{ for } s\in (0,b).
\end{align*}
 For $t\in [0,a]$, $s\in [0,b]$ and $0\leq  c \leq 1$, we have that 
 \begin{align*}
  \ds{\gamma(ct),\eta(cs)} 
  =& \ds{\gamma(ct),f\circ \gamma'(ct)} 
  + \ds{f\circ \gamma'(ct),f\circ \eta'(cs)}  
  + \ds{f\circ \eta'(cs),\eta(cs)}\\  
  \leq& A\ds{\gamma'(ct),\eta'(cs)}+3A \\
  \leq& A\left\{cE\ds{\gamma'(t),\eta'(s)} 
         + (1-c)E\ds{\gamma'(0),\eta'(0)} + C\right\} +3A\\
  \leq& A\left\{cE(A\ds{\gamma(t),\eta(s)}+3A)\right.\\
         &+ \left.(1-c)E(A\ds{\gamma(0),\eta(0)}+3A^2)+C\right\}+3A\\ 
  \leq& cA^2E\ds{\gamma(t),\eta(s)} 
         + (1-c)A^2E\ds{\gamma(0),\eta(0)} +  3A^3E+AC+3A.
  %	, where $C$ is a uiversal constant.
 \end{align*}

Finally we consider (\ref{qparam-reg})$^q$.
Let $\gamma,\eta\in \calL_Y$ and $\gamma',\eta'\in \calL_X$
be as above. Then for all $t\in [0,a]$ and $s\in [0,b]$, we have
\begin{align*}
 \abs{t-s} &\leq \theta(\ds{\gamma'(0),\eta'(0)}+\ds{\gamma'(a),\eta'(b)})\\
 &\leq \theta(A(\ds{\gamma(0),\eta(0)}+\ds{\gamma(a),\eta(b)})+6A^2).
\end{align*}
\end{proof}

The class of coarsely convex spaces is closed under direct product.
\begin{proposition}
 \label{prop:productspace} Let $(X,d_X)$ and $(Y,d_Y)$ be coarsely
 convex metric spaces.  Then the the product with the $\ell_1$-metric
 $(X\times Y, d_{X\times Y})$ is coarsely convex. Indeed let $\calL^X$
 and $\calL^Y$ be systems of good quasi-geodesic segments of
 $X$ and $Y$, respectively. Then for any quasi-geodesic segments
 $\gamma\in\calL^X$ defined on $[0,a]$ and $\eta\in\calL^Y$ defined on
 $[0,b]$, the map
 \begin{align*}
  \frac{a}{a+b}\gamma\oplus\frac{b}{a+b}\eta:[0,a+b]\ni t\mapsto
  \left(\gamma\left(\frac{a}{a+b}t\right),
  \eta\left(\frac{b}{a+b}t\right)\right) \in X\times Y
 \end{align*}
 is a quasi-geodesic segment of $X\times Y$, and 
 the family of such quasi-geodesic segments $\calL^{X\times Y}$ is 
 a system of good quasi-geodesic segments of $X\times Y$. 
\end{proposition}

\begin{proof}
 Let $X$ and $Y$ be metric spaces.
 Suppose that $X$ and $Y$ are
 $(\lambda,k,E,C,\theta,\calL^X)$-coarsely convex
 and $(\lambda',k',E',C',\theta',\calL^Y)$-coarsely convex, 
 respectively.
 It is straightforward to check that 
 the product $X\times Y$ is 
 $(\max\{\lambda,\lambda'\},k+k',\max\{E,E'\},C+C',
  \theta+\theta',\calL^{X\times Y})$-coarsely convex.
\end{proof}

%Indeed we have the following.
%\begin{proposition}
%Let $(X,d)$ and $(X',d')$ be coarsely convex metric spaces.
%Then the the product with the $\ell_1$-metric $(X\times X', d+d')$
%is coarsely convex.
%\end{proposition}
%\begin{proof}
%Let metric spaces $(X,d)$ and $(X',d')$ be
%$(\lambda,k,E,C,\theta,\mathcal L)$-coarsely convex
%and $(\lambda',k',E',C',\theta',\mathcal L')$-coarsely convex, respectively.
%Let $(X\times X', d+d')$ be the $\ell_1$-metric space.
%Then for any quasi-geodesic segments $\gamma\in\mathcal L$ defined on $[0,a]$
%and $\gamma'\in\mathcal L'$ defined on $[0,a']$,
%the map $\gamma'':[0,a+a']\ni t\mapsto
%\left(\gamma\left(\frac{a}{a+a'}t\right)
%,\gamma'\left(\frac{a'}{a+a'}t\right)\right) \in X\times X'$
%is a $(\max\{\lambda,\lambda'\},k+k')$-quasi-geodesic segment of
%$(X\times X', d+d')$.
%We denote by $\mathcal L''$ the family consisting of
%such all $(\max\{\lambda,\lambda'\},k+k')$-quasi-geodesic segments.
%Then we can check straightforwardly that
%$(X\times X', d+d')$ is
%$(\max\{\lambda,\lambda'\},k+k',\max\{E,E'\},C+C',\theta+\theta',\mathcal
%L'')$-coarsely convex.
%\end{proof}
%

CAT(0) spaces, more generally, Busemann non-positively curved spaces,
and geodesic Gromov hyperbolic spaces are examples of geodesic coarsely
convex spaces. In these examples, the set of all geodesic segments is
the system of good geodesic segments. In general, this does not
hold. Let $\Gamma_{\Z^2}$ be the Cayley graph of rank 2 free abelian
group $\Z^2$ with the standard generating set $\{(1,0), (0,1)\}$.  Let
$\gamma_n$ be a geodesic segment defined by $\gamma_n(t) := (t,0)$ for
$0\leq t \leq n$ and $\gamma_n(t):= (n,t-n)$ for $t>n$. We fix any
constant $E\geq 1$.  Then for $n\in \N$, we have
\begin{align*}
 \ds{\gamma_0(n),\gamma_n(n)}
  - \frac{1}{2E}E\ds{\gamma_0(2En),\gamma_n(2En)}
 = 2n - n = n \to \infty \quad (n\to \infty).
\end{align*}
Thus the set of all geodesic segments in $\Gamma_{\Z^2}$ does not satisfy
the condition (\ref{qconvex})$^q$ in Definition~\ref{def:cBNC}. However,
since $\Z^2$ is coarsely equivalent to $\R^2$, which is geodesic
coarsely convex, by Proposition~\ref{prop:ceq-pres-qconvity},
$\Z^2$ is coarsely convex.

\begin{example}
 Let $V$ be a normed vector space. Then $V$ is coarsely convex.
 Indeed, for $p,v\in V$ with $\norm{v}=1$, and for $r>0$, we
 define a geodesic segment $\gamma(p,v;r)\colon [0,r]\to V$ by
 $\gamma(p,v;r)(t):= p+tv$.  Let $\calLAff$ be the set of all geodesic
 segments
 $\gamma(p,v;r)$ with $p,v\in V$, $\norm{v}=1$ and $r>0$.

 Clearly $\calLAff$ satisfies~(\ref{qconn})$^q$ in
 Definition~\ref{def:cBNC}.  Since $\calLAff$ consists only of
 geodesics, it also satisfies (\ref{qparam-reg})$^q$.  For $p,v,w\in V$,
 $r,l>0$ with $\norm{v}=\norm{w} = 1$, and for $t\in [0,r]$, $s\in
 [0,l]$, $c\in [0,1]$, we have
 \begin{align*}
  \norm{(p+ctv)-(p+csw)} = c\norm{(p+tv)-(p+sw)}.
 \end{align*}
 Now it is easy to show that $\calLAff$ satisfies (\ref{qconvex})$^q$.
\end{example}

\begin{remark}
\label{rem:prefix-closed}
For a map
$\gamma\colon [a,b]\to X$, we denote by $\gamma^{-1}$, the map
$\gamma^{-1}\colon [a,b] \to X$ defined by 
$\gamma^{-1}(t):=\gamma(b-(t-a))$ for $t\in [a,b]$.  
For $c\in [a,b]$, we denote by
$\gamma|_{[a,c]}$ the restriction of $\gamma$ to $[a,c]$.
Let $\calL$ be a family of quasi-geodesic segments in $X$. 
The family $\calL$ is {\itshape symmetric} if 
$\gamma^{-1}\in \calL$ for all $\gamma\in \calL$, 
and $\calL$ is {\itshape prefix closed} if 
$\gamma|_{[a,c]} \in \calL$ for all $\gamma\in \calL$ 
with $\gamma\colon [a,b]\to X$ and for all $c\in [a,b]$.

Let $X$ be a $(\lambda,k,E,C,\theta,\calL)$-coarsely convex space.
Suppose that $\calL$ is symmetric and prefix closed. 
Then the following holds.
 Let $\gamma, \eta \in \calL$ be 
 $(\lambda,k)$-quasi-geodesic segments 
 such that $\gamma\colon [0,a]\to X$ and $\eta\colon [0,b]\to X$. 
 For $t_1,t_2\in [0,a]$, $s_1,s_2\in [0,b]$
 %$0\leq t_1, t_2\leq a$, $0\leq s_1, s_2\leq b$, 
 and $0\leq c\leq 1$, we have that 
 \begin{align*}
  \ds{\gamma(ct_2+(1-c)t_1),\eta(cs_2+(1-c)s_1)}
  \leq cE\ds{\gamma(t_2),\eta(s_2)} 
  + (1-c)E\ds{\gamma(t_1), \eta(s_1)}+ C.
  %	, where $C$ is a uiversal constant.
 \end{align*}

It seems natural to require that $\calL$ is symmetric and prefix closed
in the definition of the coarsely convex space.
%However, throughout this paper, we require neither condition.
However, in the proof of Theorem~\ref{thm:cCATHAD},
we require neither condition.
\end{remark}

%From the view point of the geometric group theory, it is quite natural to consider
%groups acting geometrically on coarsely convex spaces.

%\begin{definition}
%\label{def:c-convex-group}
% A group $G$ is {\it coarsely convex} if there exists a
% $(\lambda,k,E,C,\theta,\calL)$-coarsely convex space $X$ such that $G$
% acts on $X$ properly and cocompactly by isometries, and $\calL$ is
% invariant under the action of $G$.
%\end{definition}

%The class of coarsely convex groups includes systolic groups,
%CAT(0)-groups, and hyperbolic groups. It is natural to ask which properties of 
%these groups can be generalized to coarsely convex groups.

Finally, we mention other generalizations of the spaces of
non-positive curvature. Alonso and Bridson formulated a notion of
{\itshape semihyperbolicity} for metric spaces, and studied groups
acting on semihyperbolic spaces. In
Section~\ref{sec:groups-acting-coars}, we show that a coarsely convex
space is semihyperbolic.

Kar~\cite{Kar-asympCAT0} introduced and studied the class of metric
spaces called {\itshape asymptotically CAT(0)-spaces}. This class and
the class of coarsely convex spaces share many examples. Therefore it is
desirable to clarify the relation between these two classes of metric
spaces.

%%%%%%%%%%%%%%%%%%%%%%%%%%%%%%%%%%%%%%%%%%%%%%%%%%%%%%%%%%%%%%%%
%%%  Ideal boundary
%%%

\section{Ideal boundary}
\label{sec:ideal-boundary} Throughout this section, let X be a
$(\lambda,k,E,C,\theta,\calL)$-coarsely convex space. We will
construct the ideal boundary $\partial X$ of $X$, as the set of equivalence
classes of quasi-geodesic rays which can be approximated by
quasi-geodesic segments in $\calL$.

In this section, we introduce several constants and a function. 
Here we summarize them.
\begin{align*}
 k_1 &= \lambda + k, & D_2 &= E(D_1+2k_1), \\
 D &= 2(1+E)k_1+ C,  & D_2'&=\max\{1,E(\lambda(\theta(0))+k)\},\\
 \tilde{\theta}(t)&=\theta(t+1)+1, & D_3 &= 2D_2'(D_2)^2,\\
 D_1 &= 2D+2,& D_4 &= 2E(E(1+\lambda\tilde{\theta}(1)+2k_1)+D_1).
\end{align*}
We remark that all constants in the above list are greater than or equal
to 1. We also summarize several families of quasi-geodesic segments and
rays related to $\calL$. 

We define $\calL^\infty$ as the set of all $\calL$-approximatable maps
$\gamma:\Rp\to X$ with $\gamma(t)=\gamma(\lfloor t\rfloor)$ for all
$t\in \Rp$, where $\calL$-approximatable maps are defined in
Section~\ref{sec:approximatable-ray}.
Now let $O\in X$ be a base point. 
The following is the list of the families related to $\calL$ and 
$\calL^\infty$.
\begin{align*}
% \calL^\infty&:=\{\text{$\calL$-approximatable map }
%   \gamma:\Rp\to X \text{ with }
%  \gamma(t)=\gamma(\lfloor t\rfloor) (\forall t\in \Rp)\},\\
 \bar{\calL}&:= \calL\cup \calL^\infty, \\
 \calL_O^\infty&:=\{\gamma\in \calL^\infty: \gamma(0)=O\},\\
 \calL_O&:=\{\gamma\in \calL: 
   \gamma\colon [0,a_\gamma]\to X,\, 
   a_\gamma\geq 2\theta(0),\, \gamma(0)=O \},\\
 \bar{\calL}_O& := \calL_O\cup \calL_O^\infty. 
\end{align*}

\subsection{Approximatable ray}
\label{sec:approximatable-ray}
Let $\gamma\colon \Rp \to X$ be a
map.  Let $\gamma_n\colon [0,a_n]\to X$ be
quasi-geodesic segments in $X$.  A sequence $\{(\gamma_n,a_n)\}_{n}$ is
an {\itshape $\calL$-approximate} sequence for $\gamma$
if for all $n$, we have
$\gamma_n\in \calL$, $\gamma_n(0)=\gamma(0)$, and for all 
$l\in \N$ the sequence 
$\{\gamma_n\}_{n}$ converges to $\gamma$ uniformly on 
$\{0,1,\dots,l\}\subset \R_{\geq 0}$. 
%Here we extend the domain of $\gamma_n$ 
%on $\R_{\geq 0}$ by $\gamma_n(t)=\gamma_n(a_n)$ for all $t\geq a_n$.
%%
A map $\gamma\colon\Rp\to X$ 
is {\itshape $\calL$-approximatable} if
there exists an $\calL$-approximate sequence for $\gamma$.

\begin{lemma}
\label{lem:qg-k1}
Let $\gamma\colon \Rp\to X$ be an $\calL$-approximatable map such that
$\gamma(t)=\gamma(\lfloor t\rfloor)$ for all $t\in \Rp$.
Then $\gamma$ is a $(\lambda,k_1)$-quasi-geodesic ray, where $k_1:= \lambda+k$.
\end{lemma}

\begin{proof}
% Let $\{\gamma[k;n]\}_n$ be a family of subsequence constructed in 
%Lemma~\ref{lem:Ascoli-Arzela}.
Let $\gamma\colon \Rp\to X$ be an $\calL$-approximatable map. 
Then there exists an
$\calL$-approximate sequence $\{(\gamma_n,a_n)\}_{n}$ for $\gamma$.
 We fix $t,s\in \Rp$. Set $i:=\lfloor t\rfloor$ and 
 $j:=\lfloor s \rfloor$. Then for any $\epsilon>0$, 
 there exists an integer $n$ such that 
 \begin{align*}
  \ds{\gamma(i) , \gamma_n(i)}< \epsilon \text{ and }
  \ds{\gamma(j) , \gamma_n(j)}< \epsilon. 
 \end{align*}
 Since $\gamma_n$ is a $(\lambda,k)$-quasi-geodesic segment, we have
\begin{align*}
  \frac{1}{\lambda}\abs{i-j} - k 
  \leq \ds{\gamma_n(i),\gamma_n(j)}
  \leq \lambda\abs{i-j} + k.
\end{align*}
Then we have
 \begin{align*}
  \frac{1}{\lambda}\abs{t-s} - \frac{1}{\lambda} -k -2\epsilon
  \leq \ds{\gamma(t), \gamma(s)}
  \leq \lambda\abs{t-s} + \lambda +k +2\epsilon.
 \end{align*}
Since $\epsilon$ can be arbitrarily small, $\gamma$ is a
$(\lambda,\lambda+k)$-quasi-geodesic ray.
\end{proof}

We define a family of quasi-geodesic rays, denoted by $\calL^\infty$, 
as a family consisting of all $\calL$-approximatable maps 
$\gamma\colon \Rp\to X$ such that 
$\gamma(t)=\gamma(\lfloor t\rfloor)$ for all $t\in \Rp$.
We set $\bar{\calL}:= \calL\cup \calL^\infty$. 
Let $O\in X$ be a base point. Then we define $\calL_O^\infty$ as
the subset of $\calL^\infty$ consisting of 
all quasi-geodesic rays in $\calL^\infty$ stating at $O$.

By an argument similar to that in the proof of
Lemma~\ref{lem:qg-k1}, we have the following.
%%%%%%%%%%%%%%%%%%%%%%%%%%%%%%%%%%%%%%%%%%%%%%%%%%%%%%%%%%%%%%%%%%%%%%%%%%%%%%%%%%%%%%%%%%%%%%%%%%%%%%%
\begin{proposition}
\label{prop:qgeod-ray-convex}
Set $I=[0,a]$ or $I=\Rp$ and $J=[0,b]$ or $J=\Rp$. 
The family $\bar{\calL}$ satisfies the following.
 \begin{enumerate}[\hspace{7mm}$(1)$]
% \item All $\gamma\in \calL_O^\infty$ are $(\lambda,k_1)$-quasi-geodesic, 
%	where $k_1:=\lambda+k$.
 \item \label{qrayconvex}       
       Let $\gamma, \eta \in \bar{\calL}$ be quasi-geodesics with
       $\gamma\colon I\to X$ and $\eta\colon J\to X$.
       Then for $t\in I$, $s\in J$ and $0\leq c\leq 1$, we have 
       \begin{align*}
	\ds{\gamma(ct),\eta(cs)}
	\leq cE\ds{\gamma(t),\eta(s)} +(1-c)E\ds{\gamma(0),\eta(0)}+ D,
	%	, where $C$ is a uiversal constant.
       \end{align*}
       	where $D:= 2(1+E)k_1+ C$.
 \item \label{qrayparam-reg} We define a non-decreasing function
       $\tilde{\theta}\colon \Rp\to \Rp$ by 
       $\tilde{\theta}(t):=\theta(t+1)+1$.
       For $\gamma,\eta\in \bar{\calL}$ with $\gamma\colon I\to X$ and
       $\eta\colon J\to X$, and for $t\in I$, $s\in J$, we have
       \begin{align*}
	\abs{t-s}\leq 
	\tilde{\theta}(\ds{\gamma(0),\eta(0)}+\ds{\gamma(t),\eta(s)}).
       \end{align*}       
 \end{enumerate}
\end{proposition}

%By the same argument in the proof of Lemma~\ref{lem:Xsame-param}, we
%have the following.
\begin{lemma}
\label{lem:ray-same-param}
 Set $I=[0,a]$ or $I=\Rp$ and $J=[0,b]$ or $J=\Rp$. 
 Let $\gamma,\eta\in \bar{\calL}$ be quasi-geodesics such that
 $\gamma\colon I\to X$ and $\eta\colon J\to X$ with $\gamma(0)=\eta(0)$.
 For all $a\in I$, $b\in J$ and $0\leq t\leq \min\{a,b\}$, we have
 \begin{align*}
  \ds{\gamma(t),\eta(t)}
  &\leq E(\ds{\gamma(a),\eta(b)} 
    + \lambda\tilde{\theta}(\ds{\gamma(a),\eta(b)}) + k_1) + D.
 \end{align*} 
\end{lemma}

\begin{proof}
 We suppose $a\leq b$. Then
 \begin{align*}
  \ds{\gamma(t),\eta(t)}
  %&\leq \frac{t}{a}\ds{\gamma(a),\eta(a)} + D\\
  &\leq E\ds{\gamma(a),\eta(a)} + D\\
  &\leq E(\ds{\gamma(a),\eta(b)} + \lambda\abs{b-a} + k_1) + D\\
  &\leq E(\ds{\gamma(a),\eta(b)} 
    + \lambda\tilde{\theta}(\ds{\gamma(a),\eta(b)}) + k_1) + D.
 \end{align*}
\end{proof}

\begin{definition}
\label{def:ideal-bdry}
 For quasi-geodesic rays $\gamma$ and $\eta$ in $\calL^\infty$, 
 we say that $\gamma$ and $\eta$ are {\itshape equivalent} if
 \begin{align*}
 \sup\{\ds{\gamma(t),\eta(t)}:t\in \Rp\}<\infty,
 \end{align*}
 and we denote by $\gamma \sim \eta$. 
 For $\gamma\in \calL^\infty$, we denote by $[\gamma]$ its equivalence
 class.
 The {\itshape ideal boundary} of $X$ is the
 set $\partial X:= \calL^\infty/ \sim$ of equivalence classes of
 quasi-geodesic rays in $\calL^\infty$.
 The {\itshape ideal boundary} of $X$ with respect to $O$ is the
 set $\partial_O X:= \calL_O^\infty/ \sim$ of equivalence classes of
 quasi-geodesic rays in $\calL_O^\infty$.
\end{definition}

\begin{lemma}
\label{lem:HdistD}
 For $\gamma, \eta\in \calL_O^\infty$, if
 $[\gamma] = [\eta]$ then $\ds{\gamma(t),\eta(t)}\leq D$ for all
 $t\in \Rp$. 
\end{lemma}

\begin{proof} 
 Let $\gamma, \eta\in \calL_O^\infty$ be quasi-geodesic rays. 
 We suppose that there
 exists $s>0$ such that $\ds{\gamma(s),\eta(s)}> D$. 
 Then by Proposition~\ref{prop:qgeod-ray-convex}, for $0< c\leq 1$, 
 we have 
 \begin{align*}
  \ds{\gamma(s/c),\eta(s/c)}
  \geq \frac{1}{cE}(\ds{\gamma(s),\eta(s)} -D) \to \infty \quad (c\to 0).
 \end{align*}
 Thus we have $\sup\{\ds{\gamma(t),\eta(t)}:t\in \Rp\} = \infty$.
\end{proof}

\subsection{Gromov product}
\label{sec:gromov-product}
We define $\calL_O$ as the subset of $\calL$ 
consisting all quasi-geodesic segments $\gamma\in \calL$ with
$\gamma\colon [0,a_\gamma]\to X$, $a_\gamma\geq 2\theta(0)$ 
and $\gamma(0)=O$. Set $\bar{\calL}_O:= \calL_O\cup \calL_O^\infty$.

\begin{definition}
We define a product 
$(\,\cdot \,| \,\cdot\,)_O\colon \bar{\calL}_O\times\bar{\calL}_O 
\to \Rp\cup \{\infty\}$ 
as follows. Set $I=[0,a]$ or $I=\Rp$ and $J=[0,b]$ or $J=\Rp$. 
Then for $\gamma,\eta\in \bar{\calL}_O$ with $\gamma\colon I\to X$ and
$\eta\colon J\to X$, we define
 \begin{align*}
  (\gamma\mid \eta)_O:= 
  \sup\{t: t\in I\cap J,\, \ds{\gamma(t),\eta(t)}\leq D_1\},
 \end{align*}
where $D_1:=2D+2$.
When it is clear which
 point is the base point $O$, we write $(\gamma\mid \eta)$ instead of
 $(\gamma\mid \eta)_O$.
\end{definition}

\begin{lemma}
\label{lem:maximizer}
 Let $\gamma,\eta\in \bar{\calL}_O$ be quasi-geodesics. 
 Set $a:=(\gamma\mid \eta)$. If $a<\infty$, then
 \begin{align*}
  \ds{\gamma(a), \eta(a)} \leq D_1 +2k_1.
 \end{align*}
\end{lemma}

\begin{proof}
 Let $\gamma,\eta\in \bar{\calL}_O$ be quasi-geodesics.  Set
 $a:=(\gamma\mid \eta)$. For any positive number $\delta$ with
 $0<\delta\leq a$, 
 there exists a $\delta'$ with $0\leq \delta'\leq\delta$ and 
 $\ds{\gamma(a-\delta'), \eta(a-\delta')}\leq D_1$. 
% Since $\gamma$ and $\eta$ are $(\lambda,k)$-quasi-geodesic, 
 Now,
 \begin{align*}
  \ds{\gamma(a-\delta'),\gamma(a)}\leq \lambda\delta + k_1
  \, \text{ and } \,
  \ds{\eta(a-\delta'),\eta(a)}\leq \lambda\delta + k_1.
 \end{align*}
 Thus
 $\ds{\gamma(a), \eta(a)}\leq D_1 + 2(\lambda\delta + k_1)$.
 Since $\delta$ can be arbitrarily small, we have
 $\ds{\gamma(a), \eta(a)} \leq D_1 +2k_1$.
\end{proof}

\begin{lemma}
\label{lem:qultm}
Set $D_2:=E(D_1+2k_1)$.
 For $\gamma,\eta,\xi\in \bar{\calL}_O$, we have
 \begin{align*}
  (\gamma\mid \xi) \geq D_2^{-1}\min
  \{(\gamma\mid \eta), (\eta\mid \xi)\}.
 \end{align*}
\end{lemma}

\begin{proof}
 Set $a:= (\gamma\mid \eta)$ and $b:= (\eta\mid \xi)$. 
%We assume that $a\leq b$. 
Set $a':=  D_2^{-1} \min\{a,b\}$. Then
 \begin{align*}
  \ds{\gamma(a'),\eta(a')}&\leq \frac{a'}{a}E\ds{\gamma(a),\eta(a)}+ D
  \leq D_2^{-1}E (D_1+2k_1) + D = D+1,\\
 \ds{\eta(a'),\xi(a')}&\leq \frac{a'}{b}E\ds{\eta(b),\xi(b)}+D 
%  \leq D_2^{-1}E\ds{\eta(b),\xi(b)} + D \\
  \leq D_2^{-1}E(D_1+2k_1) + D = D+1, \\
  \ds{\gamma(a'),\xi(a')} 
  &\leq  2D+2 =D_1.
 \end{align*} 
 It follows that $(\gamma\mid \xi) \geq a'$.
\end{proof}

\begin{lemma}
\label{lem:same-rep}
 Set $D_2':=\max\{1,E(\lambda(\theta(0))+k)\}$. We have the following.
 \begin{enumerate}[\hspace{7mm}$(1)$]
  \item \label{item:seg-wt-same-end}
	For $\gamma,\eta\in \calL_O$ with $\gamma\colon [0,a]\to X$ and 
	$\eta\colon [0,b]\to X$, if $\gamma(a)=\eta(b)$, then
	\begin{align*}
	 (\gamma\mid \eta) \geq D_2'^{-1}\min\{a,b\} .
	\end{align*}
  \item \label{item:asymp-rays}
	For $\gamma,\eta\in \calL_O^\infty$, if $[\gamma]=[\eta]$, then
	\begin{align*}
	 (\gamma\mid \eta) =\infty.
	\end{align*}
 \end{enumerate}
\end{lemma}

\begin{proof}
 The statement (\ref{item:asymp-rays}) follows from Lemma~\ref{lem:HdistD}.
 Thus we show (\ref{item:seg-wt-same-end}). 
 For $\gamma,\eta\in \calL_O$ with $\gamma\colon [0,a]\to X$ and 
 $\eta\colon [0,b]\to X$, we suppose $\gamma(a)=\eta(b)$.
 Set $d:=\min\{a,b\}$. Then we have
 \begin{align*}
  \ds{\gamma(d),\eta(d)}\leq \lambda\abs{a-b}+k \leq 
  \lambda\theta(0)+k
 \end{align*}
 Then $\ds{\gamma(D_2'^{-1}d),\eta(D_2'^{-1}d)}\leq 1+D$. 
 Thus $(\gamma\mid \eta) \geq D_2'^{-1}d$.
\end{proof}

\begin{lemma}
\label{lem:D3}
 Set $D_3:=2D_2'(D_2)^2$. We have the following.
 \begin{enumerate}[\hspace{7mm}$(1)$]
  \item Let $\gamma_1, \eta_1, \gamma_2, \eta_2\in \calL_O$ 
	be quasi-geodesic segments with
	$\gamma_i\colon [0,a_i] \to X$ and $\eta_i\colon [0,b_i]\to X$ 
	for $i=1,2$.
	If $\gamma_i(a_i)=\eta_i(b_i)$ for $i=1,2$, then
  \begin{align*}
   D_3^{-1} (\gamma_1\mid \gamma_2) \leq (\eta_1\mid \eta_2) 
   \leq D_3 (\gamma_1\mid \gamma_2) .
  \end{align*}

  \item For quasi-geodesic rays 
	$\gamma_1, \eta_1, \gamma_2, \eta_2\in \calL_O^\infty$, 
	if $[\gamma_i]=[\eta_i]$ for $i=1,2$, then
	\begin{align*}
	 D_3^{-1} (\gamma_1\mid \gamma_2) \leq (\eta_1\mid \eta_2) 
	 \leq D_3 (\gamma_1\mid \gamma_2) .
	\end{align*}
  \item  Let $\gamma_1, \eta_1 \in \calL_O$ 
	 be quasi-geodesic segments with 
	 $\gamma_1\colon [0,a_1] \to X$ and $\eta_1\colon [0,b_1]\to X$.
	 Let $\gamma_2, \eta_2\in \calL_O^\infty$ be quasi-geodesic rays.
	 If $\gamma_1(a_1)=\eta_1(b_1)$ and $[\gamma_2]=[\eta_2]$ then
	 \begin{align*}
	  D_3^{-1} (\gamma_1\mid \gamma_2) \leq (\eta_1\mid \eta_2) 
	  \leq D_3 (\gamma_1\mid \gamma_2) .
	 \end{align*}	
 \end{enumerate}
\end{lemma}

\begin{proof}
 We give a proof for the first statement. 

 Since $b_i\geq 2\theta(0)$ and $\abs{a_i-b_i}\leq \theta(0)$ for $i=1,2$, 
 we have $a_i\geq b_i/2$ for $i=1,2$. 
 By Lemma~\ref{lem:qultm} and Lemma~\ref{lem:same-rep},
 \begin{align*}
  (\gamma_1\mid \gamma_2)
  &\geq 
   D_2^{-2}\min\{(\gamma_1\mid \eta_1),(\eta_1\mid\eta_2),
                   (\eta_2\mid \gamma_2)\}\\
  &\geq    (D_2'D_2^2)^{-1}\min\{a_1,b_1,(\eta_1\mid\eta_2),
                   a_2,b_2\}\\
  & \geq    (2D_2'D_2^2)^{-1}(\eta_1\mid\eta_2).
 \end{align*}
 We can prove the rest of the statement in the same way.
\end{proof}

\begin{definition}
 We define a product 
 $(\,\cdot \,| \,\cdot\,)\colon 
 (X\cup \partial_O X) \times (X\cup \partial_O X) \to \Rp$ 
 as follows.

% We define a product $(\cdot \mid \cdot)$ on $ X\cup \partial_O X$ 
% as follows.
 \begin{enumerate}%[$(1)$]
 \setcounter{enumi}{-1}
  \item For $v,w\in X\cup \partial_O X$ with 
        $v\in B_{\lambda 2\theta(0)+k}(O)$ or 
	$w\in B_{\lambda 2\theta(0)+k}(O)$,
        we define 
        \begin{align*}
 	(v\mid w):=0.
        \end{align*}
  \item For $v,w\in X\setminus B_{\lambda 2\theta(0)+k}(O)$, 
        we define 
        \begin{align*}
 	(v\mid w): = \sup (\gamma\mid \eta),
        \end{align*}
        where the supremum is taken over all $\gamma, \eta\in \calL_O$ with 
        $\gamma\colon [0,a] \to X$, $\eta\colon [0,b] \to X$, 
        $\gamma(a)=v$ and $\eta(b)=w$.
  \item For $x,y\in \partial_O X$, we define 
        \begin{align*}
 	(x\mid y): = \sup (\gamma\mid \eta),
        \end{align*}
        where the supremum is taken over all 
	$\gamma, \eta\in \calL_O^\infty$
        such that $x =[\gamma]$ and $y=[\eta]$.
  \item For $x\in \partial_O X$ and 
        $v\in X\setminus B_{\lambda 2\theta(0)+k}(O)$, we define
        \begin{align*}
 	(v\mid x):= \sup (\gamma\mid \eta),
        \end{align*}
        where the supremum is taken over all quasi-geodesic rays 
        $\eta\in \calL_O^\infty$ with $x=[\eta]$ and 
        quasi-geodesic segments $\gamma\in \calL_O$ with 
        $\gamma\colon [0,a]\to X$ and $v=\gamma(a)$. We define 
	$(x\mid v):=(v\mid x)$.
 \end{enumerate}
\end{definition}
Lemma~\ref{lem:D3} implies the following.
\begin{lemma}
\label{lem:chooseRep-full}
 We have the following.
 \begin{enumerate}[\hspace{7mm}$(1)$]
  \item For $v,w \in X\setminus B_{\lambda 2\theta(0)+k}(O)$ and for
	$\gamma, \eta\in \calL_O$ with
	$\gamma\colon [0,a] \to X$ and $\eta\colon [0,b] \to X$,
	if $\gamma(a)=v$ and $\eta(b)=w$, then
  \begin{align*}
   (\gamma\mid \eta) \leq  (v\mid w) \leq D_3(\gamma\mid \eta).
  \end{align*}
  \item For $x,y \in \partial_O X$ and for 
	$\gamma, \eta\in \calL_O^\infty$, 
	if $x=[\gamma]$ and $y=[\eta]$, then
  \begin{align*}
   (\gamma\mid \eta) \leq  (x\mid y) \leq D_3 (\gamma\mid \eta).
  \end{align*}
  \item  For $x\in \partial_O X$, 
	 $v\in X\setminus B_{\lambda 2\theta(0)+k}(O)$, and for
	 $\eta\in \calL_O^\infty$,
	 $\gamma\in \calL_O$ with 
	 $\gamma\colon [0,a]\to X$,
	 if $x=[\eta]$ and $v=\gamma(a)$, then
   \begin{align*}
    (\gamma\mid \eta) \leq  (v\mid x) \leq D_3(\gamma\mid \eta).
   \end{align*}
 \end{enumerate}
\end{lemma}

\begin{corollary}
\label{cor:qrayultm} For 
 $x,y,z\in (X\setminus B_{\lambda 2\theta(0)+k}(O))\cup \partial_O X$, 
 we have 
 \begin{align*}
  (x\mid z) \geq (D_2D_3)^{-1}\min \{(x\mid y), (y\mid z)\}.
 \end{align*}
\end{corollary}

\begin{lemma}
\label{lem:approx-seg-Gprod} 
 Let $\gamma\in \calL_O^\infty$ be a quasi-geodesic ray and let 
 $\{(\gamma_n,a_n)\}_{n}$ be an $\calL$-approximate sequence 
 for $\gamma$.  Then we have 
 $\liminf_{n\to \infty}(\gamma\mid \gamma_n)= \infty$.
\end{lemma}

\begin{proof}
 Let $\gamma\in \calL_O^\infty$ be a quasi-geodesic ray and let 
 $\{(\gamma_n,a_n)\}_{n}$ be an $\calL$-approximate sequence.
 Then for $R\in \N$, there exists $N>0$ such that for 
 all $n>N$, we have
 $\ds{\gamma(R),\gamma_n(R)}< 1\leq D_1$. 
 Thus $(\gamma\mid \gamma_n)\geq R$. 
\end{proof}

%%%%%%%%%%%%%%%%%%%%%%%%%%%%%%%%%%%%%%%%%%%%%%%%%%%%%%%%%%%%%%%
%%%
%%% pertabation of first argument in Gromov product
%%%

\begin{lemma}
\label{lem:pertarb}
Let $v,w\in X\setminus B_{\lambda 2\theta(0)+k}(O)$ be points.
Let $\eta\in \calL_O^\infty$ be a quasi-geodesic ray and
$\gamma_v,\gamma_w\in \calL_O$ be quasi-geodesic segments such that
$\gamma_v\colon [0,a_v]\to X$, 
$\gamma_w\colon [0,a_w]\to X$, $\gamma_v(a_v)=v$ 
and $\gamma_w(a_w)=w$.
Then we have 
 \begin{align}
\label{eq:2}
  (\gamma_w\mid \eta ) \geq \frac{((\gamma_v\mid \eta) - \theta(\ds{v,w}))}
      {E(E(\ds{v,w}+\lambda\tilde{\theta}(\ds{v,w})+k_1)+D_1+D+2k_1)}.  
 \end{align}
\end{lemma}

\begin{proof}
 We denote by $S$ the right hand side of (\ref{eq:2}).
 Set $a:= (\gamma_v\mid \eta)$ and $b:=\min\{a_v, a_w\}$.
 %We remark that $S\leq a$ and 
 Since $b\geq a_v-\abs{a_v -a_w}\geq (\gamma_v\mid \eta)-\theta(\ds{v,w})$,
 we have $\min\{a,b\}\geq S$. Then
 \begin{align*}
  \ds{\eta(S),\gamma_w(S)}
  &\leq \ds{\eta(S),\gamma_v(S)} + \ds{\gamma_v(S),\gamma_w(S)}\\
  &\leq \frac{S}{a}E\ds{\eta(a),\gamma_v(a)} 
      + \frac{S}{b}E\ds{\gamma_v(b),\gamma_w(b)} +2D\\
  &\leq \frac{S}{a}E(D_1+2k_1) 
  + \frac{S}{b}E(E(\ds{v,w}+\lambda\tilde{\theta}(\ds{v,w})+k_1)+D) +2D\\
  &\leq 2+2D=D_1.
 \end{align*}
% From the second line to the third, we use Lemma~\ref{lem:Xsame-param}.
 This complete the proof.
 \end{proof}

\begin{corollary}
\label{cor:pertabation}
 Set $D_4:= 2E(E(1+\lambda\tilde{\theta}(1)+k_1)+D_1+D+2k_1)$.
 For $x\in \partial X$ and 
 $v,w\in X\setminus B_{\lambda 2\theta(0)+k}(O)$, 
 if $(v\mid x)\geq 2D_3\theta(1)$ and 
 $\ds{v,w}\leq 1$, then we have
 $(w\mid x)\geq (D_3D_4)^{-1}(v\mid x)$.
\end{corollary}
%%%%%%%%%%%%%%%%%%%%%%%%%%%%%%%%%%%%%%%%%%%%%%%%%%%%%%%%%%%%%%

\subsection{Topology on \texorpdfstring{$X\cup \partial_O X$}{XudX}}
\label{sec:topon-X-dX}
%We set $\bar{X}:=X\cup \partial X$. 
For all positive integers $n\geq 1$, we set
\begin{align*}
 V_n:=&\,\{(x,y)\in \partial_O X\times \partial_O X: (x\mid y)> n\} \\
 &\cup \{(x,v)\in \partial_O X\times X: (v\mid x) > n\} \\
 &\cup \{(v,x)\in X\times \partial_O X: (v\mid x) > n\} \\
  &\cup \{(v,w)\in X\times X: (v\mid w) > n\}\\ 
 &\cup \{(v,w)\in X\times X: \ds{v,w}< n^{-1}\}. 
% U_r:=&
\end{align*}

For given $n\in \N$, we take $m\in \N_{>0}$ with
$m> D_{2}D_{3}D_{4}(\theta(1)+1)n$. 
Then by Corollary~\ref{cor:qrayultm} and Corollary~\ref{cor:pertabation}, 
for all $(p,q)\in V_m$ and $(q,r)\in V_m$, we have $(p,r)\in V_n$.
It follows that the family 
%$\{V_r\}_{r\in \Q_{>0}}\cup \{U_r\}_{r\in \Q_{>0}}$ forms 
$\{V_n\}_{n\in \N}$ forms 
a fundamental system of entourages of a uniform structure on 
$X\cup \partial_O X$. 
(see \cite[Chapter II, \S 1.1]{Bourbaki-gentop1-4}), which is 
metrizable (see \cite[Chapter IX, \S 2.4]{Bourbaki-gentop5-10}).

We remark that for $x\in X\cup \partial_O X$, the family 
$\{V_n[x]\}_{n\in \N}$ is a fundamental system of neighbourhoods of $x$.
Here $V_n[x]$ is defined by
\begin{align*}
 V_n[x]:= \{y\in X\cup \partial_O X: (x,y)\in V_n\}.
\end{align*}
For $v\in X$, if $n>\lambda(\ds{O,v}+k_1)$, then the set 
$\{(v,y)\in X\times (X\cup \partial_O X): (v\mid y) > n\}$ is empty. 
Thus $V_n[x]=\{w\in X: \ds{v,w}<n^{-1}\}$. It follows that the inclusion
$X\hookrightarrow X\cup \partial_O X$ is a topological embedding.

%We recall that a metric space $X$ is {\itshape proper} is all closed
%bounded subsets are compact.

\subsection{Construction of quasi-geodesic rays}

From now on, we always assume that the coarsely convex space $X$ is
{\itshape proper}, that is, all closed bounded subsets are compact.

For a sequence $\{v_n\}$ in $X$ which goes to infinity, we will construct a 
sequence $\{N_n\}_{n}$ in $\N$ and a sequence of quasi-geodesic segments 
$\gamma_{N_n}\in \calL_O$ connecting $O$ to $v_{N_n}$, which 
converges to a quasi-geodesic ray uniformly on every finite subsets of $\N$.
%
%Let $\{v_n\}_{n}$ be a sequence in $X$. We say that 
%$\{v_n\}_{n}$ {\itshape diverges to infinity} 
%if $\lim_{n}\ds{O,v_n}=\infty$.
%% 
%We define that $\mathcal{S}$ is the set of all sequences 
%$\{v_n\}_{n}$ which diverges to infinity.

%For a sequence $\{v_n\}_{n}\in \mathcal{S}$, 
%we will construct a quasi-geodesic ray 
%$\gamma_\infty$ associated to 
%$\{v_n\}_{n}$ as follows.

\begin{proposition}
\label{prop:Ascoli-Arzela}
 Let $\{v_n\}_{n}$ be a sequence in $X$ such that 
 $\lim_{n\to \infty}\ds{O,v_n}=\infty$.
 Then there exists a $(\lambda,k_1)$-quasi-geodesic ray 
 $\gamma\in \calL_O^\infty$ starting at $O$, and a 
 sequence $\{N_n\}$ in $\N$
 such that $\liminf_{n\to \infty}(v_{N_n}\mid [\gamma])=\infty$.
\end{proposition}

\begin{proof}
 Let $\{v_n\}_{n}$ be a sequence in $X$ such that 
 $\lim_{n\to \infty}\ds{O,v_n}=\infty$.
 We choose $(\lambda,k)$-quasi-geodesic segments  $\gamma_n:[0,a_n]\to X$
 in $\calL_O$ such that $\gamma_n(a_n)=v_n$.
 
 By induction, for all $l\geq 0$, we will construct 
 a subsequence $\{\gamma[l;n]\}_{n}$ of $\{\gamma_n\}_{n}$, 
 and a sequence $\{v_i^\infty\}_{i}$ in $X$ with $v_0^\infty=O$ 
 satisfying the following.
 \begin{enumerate}
  \item $\gamma[0;n] = \gamma_n$ for all $n\geq 0$.
  \item For all $l\geq 0$, the sequence $\{\gamma[l+1;n]\}_{n}$ is a
	 subsequence of $\{\gamma[l;n]\}_{n}$.
%  \item For all integer $0\leq i\leq l$, 
%	the sequence $\gamma[l,n](i)$ converges to $x^\infty_i$.
  \item \label{item:unif-conv}
	For all $l\geq 0$, the sequence 
	$\{\gamma[l;n]\}_n$ converges uniformly on 
	$\{0,1,\dots,l\}$ to the 
	map $\{0,1,\dots,l\}\ni i\mapsto v_i^\infty\in X$.
 \end{enumerate}

 First, we define $\gamma[0,n]=\gamma_n$ for all $n\geq 0$.
 Now we suppose that we have constructed a sequence 
 $v_0^\infty, v_1^\infty,\dots, v_l^\infty$ and a family of subsequences 
 $\{\gamma[0;n]\}_n,\dots, \{\gamma[l;n]\}_n$ satisfying the above conditions.
 Since $\gamma[l;n]$ is a $(\lambda,k)$-quasi-geodesic segment 
 for all $n\geq 0$, we have 
 \begin{align*}
  \ds{O, \gamma[l,n](l+1)}\leq \lambda(l+1)+k
 \end{align*}
 for all $n\geq 0$. 
 By the properness of $X$, 
 % every closed ball is compact, so 
 there exits a sequence $\{N^l_n\}_{n}$ of integers 
 such that the sequence $\{\gamma[l,N^l_n](l+1)\}_{n}$ converges.
 Thus we set $v^\infty_{l+1}:= \lim_{n\to \infty} \gamma[l,N^l_n](l+1)$ and 
 set $\gamma[l+1,n]:= \gamma[l,N^l_n]$.
 Then the sequence of the maps 
 $\{\gamma[l+1;n]\}_n$ converges uniformly on $\{0,1,\dots,l+1\}$ to 
 the map $i\mapsto v_i^\infty$.

 Now we define a map $\gamma\colon \Rp\to X$ by
 $\gamma(t):=v^\infty_{\lfloor t\rfloor}$. We claim that 
 for all $l\in \N$, the sequence of maps $\{\gamma[n;n]\}_{n}$ converges 
 uniformly on $\{0,1,\dots,l\}$ to the map $\gamma$.
 We fix $l\in \N$. Let $m$ be any integer with $m>l$. 
 Since $\{\gamma[m;n]\}_{n}$ 
 is a subsequence of $\{\gamma[l;n]\}_{n}$, for all $a\in \N$, 
 there exists $k(l,m,a)\in \N$ such that $\gamma[m;a]=\gamma[l;k(l,m,a)]$.
 We remark that the map $a\mapsto k(l,m,a)$ is increasing. 
 By~(\ref{item:unif-conv}), for any $\epsilon>0$, there exists 
 $n(l)\in \N$ 
 such that for all $n>n(l)$ and $i\in \{0,1,\dots,l\}$, we have 
 $\ds{v^\infty_i,\gamma[l,n](i)}<\epsilon$. Then let $n$ be an integer with 
 $n>\max\{l,n(l)\}$. Since $k(l,n,n)>n(l)$, we have
 \begin{align*}
  \ds{v^\infty_i,\gamma[n;n](i)} 
 = \ds{v^\infty_i,\gamma[l;k(l,n,n)]}< \epsilon
 \end{align*}
 for all $i\in \{0,1,\dots,l\}$. This completes the proof of the claim.

 For $n\in \N$, let $N_n$ be the integer such that 
 $\gamma_{N_n}=\gamma[n;n]$. It follows that 
 $\{(\gamma_{N_n},a_{N_n})\}_{n}$ is an $\calL$-approximate sequence for 
 $\gamma$. Thus  $\gamma\in \calL_O^\infty$ and 
 by Lemma~\ref{lem:qg-k1}, $\gamma$ is a $(\lambda,k_1)$-quasi-geodesic ray.
 By the construction, we have $\gamma_{N_n}(a_{N_n}) = v_{N_n}$.
 Since $(v_{N_n}\mid [\gamma])\geq (\gamma_{N_n}\mid \gamma)$, by 
 Lemma~\ref{lem:approx-seg-Gprod}, we have 
 $\liminf_{n\to \infty}(v_{N_n}\mid [\gamma])=\infty$.
\end{proof}

\begin{proposition}
 \label{lem:dXcptmet} 
 For a proper coarsely convex space $X$, the uniform space 
 $X\cup \partial_{O} X$ is compact.
\end{proposition}

\begin{proof}
 Since $X\cup \partial_{O} X$ is metrizable, it is enough to show that
 every infinite sequence of points has a convergent
 subsequence.  Let $\{p_n\}_n$ be a sequence in 
 $X\cup \partial_{O} X$.  By choosing
 a subsequence, we can assume either of the following holds.
 \begin{enumerate}[(a)]
 \item \label{seqInX}
       $p_n\in X$ for all $n$.
 \item \label{seqIndX}
       $p_n\in \partial_O X$ for all $n$.
 \end{enumerate}
 First we consider the case (\ref{seqInX}). 
 We can suppose $\lim_n\ds{O,p_n}=\infty$.
 By Proposition~\ref{prop:Ascoli-Arzela}, there exists a quasi-geodesic ray
 $\gamma\in \calL_O^\infty$, and a sequence $\{N_n\}$ in $\N$ such that 
 $\liminf(p_{N_n}\mid [\gamma]) = \infty$. This shows that the subsequence
 $\{p_{N_n}\}$ converges to $[\gamma]$. 

 Next we consider the case (\ref{seqIndX}). 
 We choose quasi-geodesic rays $\eta_n\in \calL_O^\infty$ such that
 $p_n=[\eta_n]$. For each $n\in\N$, we set $v_n:=\eta_n(n)$.

 Let $\eta'_n\in \calL_O$ be a quasi-geodesic segment such that 
 $\eta'_n\colon [0,a_n]\to X$ and $\eta'_n(a_n)=v_n$. 
 Since 
 $\ds{\eta'_n(a_n),\eta_n(a_n)}= \ds{\eta_n(n),\eta_n(a_n)}
  \leq \lambda(\tilde{\theta}(0))+k_1$, 
 we have 
 \begin{align*}
% \label{eq:vn-etan}
  (v_n\mid p_n)\geq (\eta'_n\mid \eta_n) 
  \geq \frac{a_n}{E(\max\{\lambda(\tilde{\theta}(0))+k_1,1\})}
  \geq 
  \frac{\ds{O,v_n}-k}{\lambda E(\lambda(\tilde{\theta}(0))+k_1+1)} 
  \to \infty.
 \end{align*}

 By Proposition~\ref{prop:Ascoli-Arzela}, there exists a quasi-geodesic ray
 $\gamma\in \calL_O^\infty$, and a sequence 
 $\{N_n\}$ in $\N$ such that $\liminf(v_{N_n}\mid [\gamma]) = \infty$.
 By Lemma~\ref{cor:qrayultm}, we have
 \begin{align*}
  (p_{N_n}\mid [\gamma])
  \geq (D_2D_3)^{-1}\min\{(v_{N_n}\mid p_{N_n}), (v_{N_n}\mid [\gamma])\}
  \to \infty.
 \end{align*}
 This shows that the subsequence $\{p_{N_n}\}$ converges to $[\gamma]$.
\end{proof}

%%%%%%%%%%%%%%%%%%%%%%%%%%%%%%%%%%%%%%%%%%%%%%%%%%%%%%%%%%%%%%%%%%%%
%%%
%%% Metric on the ideal boundary
%%%

\subsection{Metric on the ideal boundary}
%\begin{definition}
% Let $X$ be a proper geodesic coarsely convex space. 
Let $\epsilon> 0$
 be a positive number. For $x,y \in \partial_O X$, we define
 $\qds (x,y):= (x\mid y)^{-\epsilon}$.
%\end{definition}
%
It immediately follows that
\begin{enumerate}
 \item for $x,y \in \partial_O X$, we have $\qds(x,y) = 0$ if and
       only if $x=y$,
 \item for $x,y \in \partial_O X$, we have $\qds(x,y) =\qds(y,x)$,
 \item for $x,y,z\in \partial_O X$, we have 
       \begin{align*}
	\qds(x,z)\leq 
       (D_2D_3)^{\epsilon} \max\{\qds(x,y), \qds(y,z)\}.
       \end{align*}
\end{enumerate}
Therefore, $\qds$ is a {\itshape quasi-metric}. There exists a
standard method, so called the {\itshape chain construction}, to obtain a
metric which is equivalent to $\qds$. For detail, see \cite{MR2268484}.

\begin{proposition}
\label{prop:q-met2met}
 Let $\epsilon$ be a positive number such that
 $(D_2D_3)^{\epsilon}\leq 2$. 
 Then there exists a metric $d_\epsilon$ on $\partial X$ such that
 $1/(2K)\qds(x,y)\leq d_\epsilon(x,y)\leq \qds(x,y)$ 
 for all $x,y\in \partial_O X$, where $K:= (D_2D_3)^{\epsilon}$.
\end{proposition}

%\begin{corollary}
% The topology on $\partial X$ induced by the metric
% $d_\epsilon$ coincides with the topology induced by the uniform structure.
%\end{corollary}

%%%%%%%%%%%%%%%%%%%%%%%%%%%%%%%%%%%%%%%%%%%%%%%%%%%%%%%%%%%%%%%%
%%%
%%% choice of the base point

\subsection{Replacement of the base point}
\label{sec:choice-base-point}
%Let $O'\in X$ be another base point. 
Here we will construct a map 
$\Phi_{O}\colon \partial X \to \partial_O X$.

Let $\gamma\in \calL^\infty$ be a quasi-geodesic ray.  Set
$v_n:= \gamma(n)$ for $n\in \N$.  By Proposition~\ref{prop:Ascoli-Arzela},
there exists a quasi-geodesic ray $\gamma_O\in \calL_O^\infty$ 
and a sequence $\{N_n\}_{n}$ in $\N$ such that 
\begin{align*}
 \liminf(v_{N_n}\mid [\gamma_O])=\infty.
\end{align*}
%$\liminf(v_{N_n}\mid [\gamma])=\infty$.

We define a map 
$\Phi_{O}\colon \partial X \to \partial_{O} X$ 
by $\Phi_{O}([\gamma]):=[\gamma_O]$
for $\gamma\in \calL^\infty$. %By Lemma~\ref{lem:chg-base-pt}, 
By the following lemma, the map $\Phi_{O}$ is well-defined.

\begin{lemma}
\label{lem:chg-base-pt}
%We define a constant $C_{OO'}$ as follows. 
% \begin{align*}
%  C_{OO'}:= E\ds{O,O'}+k_1+D+ 2(1+\lambda+k)
% \end{align*}
Let $O'\in X$ be a point. 
For a quasi-geodesic ray $\gamma\in\calL_{O'}^\infty$ starting at $O'$, 
we have 
 \begin{align*}
  \sup\{\ds{\gamma(t),\gamma_O(t)}:t\in \Rp\}
  \leq C_{OO'},
 \end{align*}
where $C_{OO'}:= E(\lambda(\tilde{\theta}(\ds{O,O'}))+\ds{O,O'}+D_1+3k_1)+2D$.
\end{lemma}

\begin{proof}
Let $\gamma\in \calL_{O'}^\infty$ be a quasi-geodesic ray.  Set
$v_n:=\gamma(n)$ for $n\in \N$.  By
Proposition~\ref{prop:Ascoli-Arzela}, for any $R>0$ there exists $N$
such that $(v_N\mid [\gamma_O])\geq RD_3$.  Let $\gamma_N\in \calL_O$ be a
quasi-geodesic segment such that $\gamma_N\colon [0,a_N]\to X$ and
$\gamma_N(a_N)=v_N$.  Set $a:= (\gamma_N\mid \gamma_O)$. Then 
$a\geq R$ by Lemma~\ref{lem:chooseRep-full}, and 
$\ds{\gamma_N(a),\gamma_O(a)}\leq D_1 + 2k_1$ 
by Lemma~\ref{lem:maximizer}. Thus it
follows that for all $t\in [0,R]$, we have

\begin{align}
\label{eq:gN-gO}
 \ds{\gamma_N(t),\gamma_O(t)}
 \leq   E(\ds{\gamma_N(a),\gamma(a)}) + D \leq  E(D_1+2k_1)+D.
\end{align}

Since $\gamma_N(a_N)=v_{N}=\gamma(N)$, for all $t\in [0,a_N]$,
\begin{align*}
 \ds{\gamma(t),\gamma_N(t)}
 &\leq  E(\ds{\gamma(a_N),\gamma_N(a_N)}+ \ds{O,O'}) + D\\
 &\leq E(\ds{\gamma(N),\gamma_N(a_N)}+ \lambda\abs{N-a_N}+k_1+ \ds{O,O'})+D\\
 &\leq E(\lambda(\tilde{\theta}(\ds{O,O'}))+k_1+ \ds{O,O'})+D.
\end{align*}
Combined with (\ref{eq:gN-gO}), we have 
$\ds{\gamma(t),\gamma_O(t)}\leq C_{OO'}$ for all $t\in [0,R]$.
Since $R$ is arbitrary, we complete the proof of the Lemma.
\end{proof}

\begin{corollary}
\label{cor:BOO-bij}
 The map $\Phi_{O}\colon \partial X \to \partial_{O} X$ 
 is bijective.
\end{corollary}
We equip $\partial X$ with the topology such that the map $\Phi_{O}$ 
is a homeomorphism. This topology does not depend on the choice of $O$.
Indeed, by following lemmae, we can show that
the composite $\Phi_{OO'}:=\Phi_{O}\circ \Phi_{O'}^{-1}$ is continuous.
\begin{lemma}
\label{lem:chg-base-pt-prod}
 Set $D_{OO'}:=E(E(D_1+2k_1)+D + 2C_{OO'})$.
 For $\gamma,\eta\in \calL_{O'}^\infty$, we have
 \begin{align*}
  (\gamma_O\mid \eta_O)_{O}
  \geq D_{OO'}^{-1}(\gamma\mid \eta)_{O'}.
 \end{align*}
\end{lemma}

\begin{proof}
 Let $t>0$ be any positive number with $(\gamma\mid \eta)_{O'}> t$. 
 Then we have
 $\ds{\gamma(t),\eta(t)}\leq E(D_1+2k_1) + D$. 
 By Lemma~\ref{lem:chg-base-pt}, we have
 \begin{align*}
  \ds{\gamma_O(t),\eta_O(t)}
  &\leq \ds{\gamma_O(t),\gamma(t)}  
      +\ds{\gamma(t),\eta(t)} 
      +\ds{\eta(t),\eta_O(t)}\\
  &\leq E(D_1+2k_1)+D + 2C_{OO'}.
 \end{align*}
 Then 
 $\ds{\gamma_O(D_{OO'}^{-1}t),\eta_O(D_{OO'}^{-1}t)}\leq D+1\leq D_1$. 
 Thus we have 
 $(\gamma_O\mid \eta_O)_{O}\geq D_{OO'}^{-1}t$. Since $t$ is any
 positive number with $(\gamma\mid \eta)_{O'}\geq t$, we have
 $(\gamma_O\mid \eta_O)_{O}\geq D_{OO'}^{-1}(\gamma\mid \eta)_{O'}$.
\end{proof}

By the same argument in the proof of the above lemma, we have the following.
\begin{lemma}
\label{lem:chg-base-pt-prod2}
 For $\gamma\in \calL_{O'}^\infty$ and $v\in X$, we have
 \begin{align*}
  (\gamma_O\mid v)_{O} \geq D_3^{-1}D_{OO'}^{-1}(\gamma\mid v)_{O'}.
 \end{align*}
\end{lemma}

\begin{corollary}
 The map 
$\Phi_{OO'}\colon X\cup \partial_{O'} X \to X\cup \partial_{O}X$ 
 defined as an extension by the identity on $X$ of the map
 $\Phi_{OO'}\colon \partial_{O'} X \to \partial_{O}X$ is a homeomorphism.
\end{corollary}

\begin{proof}
 By Corollary~\ref{cor:BOO-bij},
$\Phi_{OO'}$ is a bijection between the compact metrizable spaces.
By Lemma~\ref{lem:chg-base-pt-prod} and
Lemma~\ref{lem:chg-base-pt-prod2}, the map $\Phi_{OO'}$ is continuous,
therefore it is a homeomorphism.
\end{proof}

\begin{corollary}
\label{cor:action-on-dX} Let $G$ be a group and let $X$ be a
 $(\lambda,k,E,C,\theta,\calL)$-coarsely convex space. We suppose that
 $G$ acts on $X$ properly and cocompactly by isometries, and $\calL$ is
 invariant under the action of $G$. Then the action of $G$ extends
 continuously to the ideal boundary $\partial X$.
\end{corollary}

%%% Examples %%%
\subsection{Examples}
Let $X$ be a proper geodesic Gromov hyperbolic space and let $\calL$ be
a set of all geodesic segments. Then $X$ is a coarsely convex
space with the system of good geodesic segments $\calL$. 
The Gromov boundary of
$X$ is homeomorphic to the ideal boundary $\partial X$. In fact, the
Gromov boundary is identical to $\partial X$ as a set. It is easy to
show that usual topology of the Gromov boundary coincides with the one
given in Section~\ref{sec:topon-X-dX}.

Next we consider the ideal boundary of the Euclidean plane
$\R^2$. Let $\calL_{\R^2}$ be a set of all geodesic segments in
$\R^2$. Then $\R^2$ is a coarsely convex space with the system of good
geodesic segments $\calL_{\R^2}$. We consider the visual
compactification of $\R^2$. Namely, we define an embedding
$\varphi\colon \R^2\to D^2=\{v\in \R^2: \norm{v}\leq 1\}$ by
$\varphi(v)= v/(1+\norm{v})$ for $v\in \R^2$. We can
identify the ideal boundary $\partial \R^2$ with $S^1\subset D^2$ as a
set.  For $x\in S^1\subset D^2$, we define a geodesic ray $\eta_x\colon
\R_{\geq 0}\to \R^2$ by $\eta_x(t)=tx$. Now for $x,y\in S^1$, let
$\theta$ be the angle between $\eta_x$ and $\eta_y$. Then we have
\begin{align*}
 \sin \frac{\theta}{2} = \frac{D}{2(x\mid y)}
\end{align*}
where $D$ is a constant defined in 
Proposition~\ref{prop:qgeod-ray-convex}~(\ref{qrayconvex}).
This shows that the topology on $\partial \R^2$ coincides with that of $S^1$.

\begin{figure}[htbp]
   \centering
     \includegraphics[page=1]{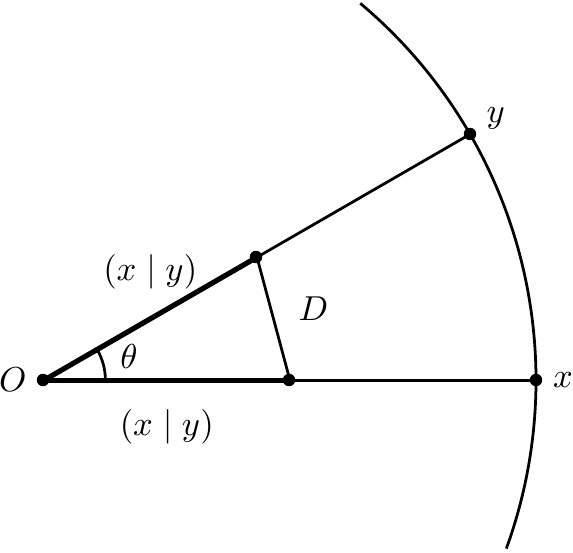}
  \caption{The ideal boundary of $\R^2$}
  \label{fig:tree}
\end{figure}

As mentioned in Proposition~\ref{prop:productspace}, a direct product of
coarsely convex spaces is also coarsely convex. The ideal boundary of
the product space is given by the join of the ideal
boundaries of the factors.

Here we recall the definition of the join.
Let $W_1$ and $W_2$ be topological spaces.
Then we consider the following equivalence relation $\sim$
on a space $W_1\times [0,1]\times W_2$.
If $(w_1,s,w_2), (w'_1,s',w'_2)$ satisfy one of three conditions
\begin{enumerate}
 \item $w_1=w'_1, s=s', w_2=w'_2$,
 \item $w_1=w'_1, s=s'=0$,
 \item $s=s'=1, w_2=w'_2$,
\end{enumerate}
then they are equivalent, that is, $(w_1,s,w_2)\sim (w'_1,s',w'_2)$.
We call the quotient space 
$W_1\star W_2:=W_1\times [0,1]\times W_2/\sim$
the join of $W_1$ and $W_2$, and denote by $(1-s)w_1\oplus sw_2$
the element whose representative is $(w_1,s,w_2)$.

%If $W_1$ and $W_2$ are topological spaces, then $W_1\*W_2$
%equipped with the quotient topology is called the topological join of
%$W_1$ and $W_2$.

\begin{proposition}
 Let metric spaces $(X,d_X)$ and $(Y,d_Y)$ be coarsely convex with
 systems of good quasi-geodesic segments $\calL^X$ and
 $\calL^Y$, respectively.  Take base points $O_X\in X$ and $O_Y\in Y$.
 For the product with the $\ell_1$-metric $(X\times Y, d_{X\times Y})$,
 which is coarsely convex with $\calL^{X\times Y}$ defined as in
 Proposition~\ref{prop:productspace}, the boundary $\partial_{(O_X,O_Y)}
 (X\times Y)$ is homeomorphic to the join
 of $\partial_{O_X} X$ and $\partial_{O_Y}Y$.
\end{proposition}
\begin{proof}
 Put $Z=X\times Y$ and $O_Z=(O_X,O_Y)$.  Let $(\calL^X_{O_X})^\infty$,
 $(\calL^Y_{O_Y})^\infty$ and $(\calL^Z_{O_Z})^\infty$ be the families
 of all approximatable rays of $\calL^X$, $\calL^Y$ and
 $\calL^Z$ from base points, respectively.  Note that these families
 does not change if we replace $\calL^X$, $\calL^Y$ and
 $\calL^Z$ with their prefix-closures, respectively. We assume that
 $\calL^X$, $\calL^Y$ and $\calL^Z$ are prefix-closed without loss of
 generality.  %Take the natural surjections
 Take the quotient maps
 \begin{align*}
  \pi_X\colon (\calL^X_{O_X})^\infty\to  \partial_{O_X} X, \quad 
  \pi_Y\colon (\calL^Y_{O_Y})^\infty\to \partial_{O_Y} Y,\quad
  \pi_Z\colon (\calL^Z_{O_Z})^\infty\to \partial_{O_Z} Z.
 \end{align*}
 We have the natural map between joins 
 \begin{align*}
  \pi_X\star \pi_Y\colon
  (\calL^X_{O_X})^\infty\star(\calL^Y_{O_Y})^\infty
  &\to 
  \partial_{O_X} X\star\partial_{O_Y}Y \\
  (1-s)\gamma\oplus s\eta
  &\mapsto (1-s)[\gamma]\oplus s[\eta].
 \end{align*}

 Now we consider the map
 \begin{align*}
  \iota\colon (\calL^X_{O_X})^\infty\star(\calL^Y_{O_Y})^\infty
  \to
  (\calL^Z_{O_Z})^\infty
 \end{align*}
 defined as $\iota((1-s)\gamma\oplus s\eta)(t)=(\gamma((1-s)t),\eta(st))$ 
 for any $t\in\Rp$. 
 This is well-defined. Indeed for any 
 $(1-s)\gamma\oplus s\eta$, we can take $\gamma_i\in\calL^X$ 
 with domain
 $[0,a_i]$ and $\eta_i\in\calL^Y$ with domain $[0,b_i]$
 such that they approximate $\gamma$ and $\eta$, respectively, and
 satisfy $s=\frac{b_i}{a_i+b_i}$ by noting prefix-closedness of
 $\calL^X$ and $\calL^Y$.  Then $\iota((1-s)\gamma\oplus s\eta)$ is
 approximated by $(1-s)\gamma_i\oplus s\eta_i\in \calL^Z$. 

 Now we can see that $\iota$ induces the map
 \begin{align*}
   \bar{\iota}\colon \partial_{O_X}
   X\star\partial_{O_Y}Y
  &\to\partial_{O_Z}Z \\
  (1-s)[\gamma]\oplus s[\eta]
  &\mapsto [(1-s)\gamma\oplus s\eta],
 \end{align*}
 which satisfies
 $\pi_Z\circ\iota=\bar{\iota}\circ (\pi_X\star\pi_Y)$ and is a
 homeomorphism.
\end{proof}

%%%%%%%%%%%%%%%%%%%%%%%%%%%%%%%%%%%%%%%%%%%%%%%%%%%%%%%%%%%%%%%%

\section{Main result}
\label{sec:main-result} The aim of this section is to give a proof of
Theorem~\ref{thm:cCATHAD}. An outline of the proof is parallel to that
for the case of Gromov hyperbolic spaces by Higson and Roe~\cite{MR1388312}.
However, we need to modify the arguments in order to overcome
some difficulties which do not appear in the case of Gromov
hyperbolic spaces.

Here we summarize the strategy used in~\cite{MR1388312}.
Let $Y$ be a proper geodesic Gromov hyperbolic space.
Higson and Roe defined an {\itshape ``exponential map''} 
$\exp \colon \calO \partial_O Y\to Y$. They first constructed 
a coarse homotopy between the open cone $\calO \partial_O Y$
and the image of the exponential map $\exp(\calO \partial_O Y)$. 
Then they constructed a coarse homotopy between $\exp(\calO \partial_O Y)$
and $Y$. Here they used the fact that the image $\exp(\calO \partial_O Y)$
is quasi-convex and the nearest point projection onto
$\exp(\calO \partial_O Y)$ is bornologous. 
In fact, $\exp(\calO \partial_O Y)$ is a 
{\itshape ``coarsely deformation retract''} of $Y$.

Now let $X$ be a proper coarsely convex space. We introduce a {\itshape
modified} exponential map $\expe \colon \OdoX\to X$ 
by replacing the parameter $t$ by $t^{\frac{1}{\epsilon}}$.
% to the power of $1/\epsilon$. 
We first construct a coarse
homotopy between $\OdoX$ and $\expe(\OdoX)$.

Then we construct a coarse homotopy between $\exp(\calO \partial_O X)$
and $X$. Here we need quite different arguments, since the image
$\expe(\OdoX)$ is not quasi-convex, and the nearest points projection is
not bornologous, in general.  In Section~\ref{sec:c-homot-betw-exp-X},
we construct the coarse homotopy using a contraction toward the
base point with an appropriate proportion, which is not necessarily a
coarsely deformation retract.

\subsection{Setting}
Let $X$ be a proper 
$(\lambda,k,E,C,\theta,\calL)$-coarsely convex space.
We fix a positive number $0 < \epsilon < 1$ such that
$(D_2D_3)^{\epsilon}\leq 2$, where $D_2$ and $D_3$ are constants
defined in Section~\ref{sec:ideal-boundary}. 
We set $K:=(D_2D_3)^{\epsilon}$.
Let $d_\epsilon$ be a metric given by
Proposition~\ref{prop:q-met2met}.  
We remark that the diameter of $(\partial_O X,d_\epsilon)$ is less than 
or equal to 1 since $(\gamma\mid \eta)\geq 1$ for all 
$\gamma,\eta\in \calL_O^\infty$. Thus the induced metric 
$d_{\OdoX}$ on the open cone $\OdoX$ is well-defined.

%In this section, replacing $X$ by a uniformly discrete $A$-dense subset
%for some $A>0$, we can suppose without loss of generality that $X$ 
%is uniformly 1-discrete, that is, for all $v,w\in X$, if $v\neq w$, then
%$\ds{v,w}\geq 1$.

%Thus we can assume without loss of generality that $M=0$, that is,
%$X=X$.
%%
\subsection{Exponential map}
\label{sec:exponential-map}
%For a positive number $\epsilon>0$, 
We define an exponential map $\expe \colon \OdoX \to X$ as
follows.  For each $x\in \partial_O X$, we choose a quasi-geodesic ray
$\eta_x\in \calL_O^\infty$ with $x=[\eta_x]$. 
Then for $t\in \Rp$, we define 
$\expe(tx):=\eta_x(t^{\frac{1}{\epsilon}})$, 
%%
%where $\epsilon$ is a positive number given in
%Proposition~\ref{prop:q-met2met}.  
We remark that $\expe$ is proper,
however, not necessarily bornologous.  Therefore, we need to modify
$\expe$ by combining with a {\itshape radial contraction}.

\begin{definition}
 Let $r\colon \Rp \to \Rp$ be a Lipschitz map with
 Lipschitz constant less than or equal to 1 such that $r(0)=0$ and
 $r(t)\to \infty$ when $t\to \infty$. The radial contraction associated
 to $r$ is a map $\radcon\colon \OdoX\to \OdoX$ defined by
 $\radcon(tx):= r(t)x$ for $tx\in \OdoX$.
\end{definition}
We remark that any radial contraction is coarsely homotopic to the identity.
\begin{definition}
\label{def:pseudocontinous}
 Let $V$ be a topological space and $M$ be a metric space. A map 
 $f\colon V\to M$ is {\itshape pseudocontinuous} if there exists $r>0$
 such that for any  $x\in V$, the inverse image
 $f^{-1}(B_r(f(x)))$ is a neighborhood of $x$.
 Here $B_r(f(x))$ is the closed ball of radius $r$ centered at $f(x)$ 
\end{definition}

\begin{proposition}[Higson-Roe\cite{MR1388312}]
\label{prop:psuedconti-contraction}
 Let $f\colon \OdoX\to X$ be a proper pseudocontinuous map.
 There exists a radial contraction $\radcon\colon \OdoX\to \OdoX$ 
 such that the composite $f\circ \radcon$ is a coarse map.
\end{proposition}
For proof, see \cite[Lemma 4.2]{MR1388312} or \cite[4.7.5]{WillettThesis}.

\begin{lemma}
\label{lem:pseudoconti}
 The map $\expe\colon \OdoX \to X$ is pseudocontinuous.
\end{lemma}
\begin{proof}
 Since the map $\expe$ is a composite of the continuous map
$tx\mapsto t^{\frac{1}{\epsilon}} x$ and $\exp_1$, it is enough to show that
$\exp_1$ is pseudocontinuous. 

For $tx\in \OdoX$ with $0\leq t\leq 1$, a neighbourhood 
$\{sy\in \OdoX: 0\leq s < 2,\, y\in \partial_O X\}$ 
of $tx$ is contained in $\exp_1^{-1}(B_{3\lambda + 2k_1}(\exp_1(tx)))$.

Thus we will show that for $x,y\in X$ and 
$t,s\in [1,\infty)$, if 
$d_{\OdoX }(tx,sy)< (2KD_3^\epsilon)^{-1}$, 
then $\ds{\exp_1(tx),\exp_1(sy)}\leq E(D_1+2k_1) + D +\lambda+k_1$. 

We take quasi-geodesic rays $\eta_x,\eta_y$ as in the definition of the
exponential map. We assume that $s\geq t$. Since 
$d_{\OdoX}(tx,sy)= \abs{s-t} + td_\epsilon(x,y)< (2KD_3^\epsilon)^{-1}$, 
we have $\abs{s-t} < 1$ and
$td_\epsilon(x,y)<(2KD_3^\epsilon)^{-1}$. 
Set $a:= (\eta_x\mid \eta_y)$. We have
\begin{align*}
 a = (\eta_x\mid \eta_y) \geq D_3^{-1}(x\mid y) 
\geq D_3^{-1}(2Kd_\epsilon(x,y))^{-\frac{1}{\epsilon}}
 > t^{\frac{1}{\epsilon}}\geq t.
\end{align*}
Therefore 
\begin{align*}
 \ds{\exp_1(tx),\exp_1(sy)} 
 &\leq \ds{\eta_x(t),\eta_y(t)} + \ds{\eta_y(t),\eta_y(s)}\\
 &\leq E\ds{\eta_x(a),\eta_y(a)} + D + \lambda\abs{s-t} + k_1 \\
 &\leq E(D_1+2k_1) + D +\lambda + k_1.
\end{align*}
\end{proof}

\begin{corollary}
\label{prop:contraction} 
 There exists a radial contraction
 $\radcon\colon \OdoX\to \OdoX$ such that the composite 
 $\expe\circ \radcon$ is a coarse map.
\end{corollary}

\subsection{Logarithmic map}
We define a logarithmic map 
\begin{align*}
 \loge \colon \expe(\OdoX) \to \OdoX
\end{align*}
as follows.  For $v\in \expe(\OdoX)$, we choose a geodesic ray
$\gamma_v\in \calL_O^\infty$ and a parameter $t_v\in \Rp$ such that 
$\gamma_v(t_v) = v$. Then we define $\loge(v):=t^{\epsilon}[\gamma_v]$.

\begin{proposition}
 \label{prop:log-coarse} The logarithmic map $\loge \colon
 \expe(\OdoX) \to \OdoX$ is a coarse map.
\end{proposition}

\begin{proof}
 It is easy to see that the map $\loge$ is proper, thus we will show that
 it is bornologous.

Let $v,w \in \expe(\OdoX)$. 
We take quasi-geodesic rays $\gamma_v,\gamma_w\in \calL_O^\infty$ and
parameters $t_v,t_w\in \Rp$ as in the definition of the map
$\loge$. Set $T:=\min\{t_v,t_w\}$.  
Fist we suppose that $T< 1$. Then we have 
$d_{\OdoX}(\loge(v),\loge(w)) \leq t_v^\epsilon + t_w^\epsilon
<1+(1+\tilde{\theta}(\ds{v,w}))^\epsilon$.

Now we suppose that $T\geq 1$. Then by an elementary calculus,
\begin{align}
\label{eq:tvtw}
 \abs{t_v^\epsilon - t_w^\epsilon} \leq 
 \epsilon\abs{t_v-t_w}\leq \epsilon\tilde{\theta}(\ds{v,w}).
\end{align}
%%
% We can assume without loss of generality that
%$t_v = t_w$. Set $T:=t_v=t_w$. 
%By Lemma~\ref{lem:ray-same-param}, we have 
%$\abs{t_v - t_w}\leq \tau(\ds{v,w})$.
%%
Now we will show that
\begin{align}
\label{eq:Gprod}
  (\gamma_v\mid \gamma_w)%\geq D_3(\gamma_v\mid\gamma_w)
   \geq \frac{T}{E\tau(\ds{v,w})}
\end{align}
where $\tau\colon \R\to \R$ is an increasing map defined by 
$\tau(t):=E(t+\lambda\tilde{\theta}(t)+k_1)+D$.

If $T\leq (\gamma_v\mid\gamma_w)$ then (\ref{eq:Gprod}) immediately follows.
Thus we suppose that $T> (\gamma_v\mid\gamma_w)$. 
By Lemma~\ref{lem:ray-same-param}, we have 
$\ds{\gamma_v(T),\gamma_w(T)}\leq \tau(\ds{v,w})$.
Set $c:=(E\tau(\ds{v,w}))^{-1}$. Then
\begin{align*}
 \ds{\gamma_v(cT),\gamma_w(cT)}
 \leq  D+1.
\end{align*}
Thus $ (\gamma_v\mid \gamma_w)\geq cT=T(E\tau(\ds{v,w}))^{-1}$.
Combined with~(\ref{eq:tvtw}) and (\ref{eq:Gprod}),
\begin{align*}
 d_{\OdoX}(\loge(v),\loge(w)) 
 &= \abs{t_v^\epsilon-t_w^\epsilon} 
 + T^{\epsilon}d_\epsilon([\gamma_v],[\gamma_w])
 \leq \epsilon\tilde{\theta}(\ds{v,w}) 
 + T^{\epsilon}(\gamma_v\mid \gamma_w)^{-\epsilon}\\
 &\leq \epsilon\tilde{\theta}(\ds{v,w}) +
 \left(E\tau(\ds{v,w})\right)^\epsilon.
\end{align*}
\end{proof}

\subsection{Coarse homotopy between \texorpdfstring{$\OdoX$}{OdoX} and 
\texorpdfstring{$\expe(\OdoX)$}{expe(OdoX)}}

\begin{lemma}
\label{lem:log-exp-chmtp-idOdX}
 The composite $\loge \circ \expe\circ \radcon$ 
 is coarsely homotopic to the identity $\mathrm{id}_{\OdoX}$.
\end{lemma}

\begin{proof}
 Since the radial contraction $\radcon$ is coarsely homotopic to the
 identity, it is enough to show that $\loge \circ \expe$ is close to the
 identity.

 For $x\in \partial_O X$, let $\eta_x\in \calL_O^\infty$ be the
 quasi-geodesic ray representative for $x$ chosen in the definition of
 the exponential map. Thus $x=[\eta_x]$.
 For $t\in \Rp$, set
 $v:= \expe(tx) = \eta_x(t^{\frac{1}{\epsilon}})$. 

 Let $\gamma_v\in \calL_O^\infty$ and $t_v\in\Rp$ 
 be the quasi-geodesic ray and the parameter,
 respectively, associated to $v$ chosen in the definition of the
 logarithmic map. 
 Thus we have
% $\ds{\eta_x(t^{\frac{1}{\epsilon}}),\gamma_v(t_v)}\leq D_4$.
 $\gamma_v(t_v) = \eta_x(t^{\frac{1}{\epsilon}})$
 and
 \begin{align*}
 \loge \circ \expe(tx) = \loge(v) = t_v^\epsilon[\gamma_v].
 \end{align*}

 Set $a:= \min\{t^{\frac{1}{\epsilon}},t_v\}$. 
 By Lemma~\ref{lem:ray-same-param}, we have 
 $\ds{\eta_x(a),\gamma_v(a)}\leq E(\lambda\tilde{\theta}(0)+k_1)+D$.
 Set $c:=(E^2(\lambda\tilde{\theta}(0)+k_1)+DE)^{-1}$.
%%%%%%%%%
% without using Lemma 
% \begin{align}
%  \label{eq:1}
%  \ds{\eta_x(a),\gamma_v(a)} \leq \lambda\theta(0)+k.
% \end{align}
% Indeed, if $a=t^{\frac{1}{\epsilon}}$, then 
% \begin{align*}
%  \ds{\eta_x(a),\gamma_v(a)}\leq 
%\ds{\eta_x(t^{\frac{1}{\epsilon}}),\gamma_v(t)}
%+\ds{\gamma_v(t),\gamma_v(t^{\frac{1}{\epsilon}})}
%\leq \lambda\absline{t^{\frac{1}{\epsilon}}-t_v} 
% +k \leq \lambda\theta(0)+k.
% \end{align*}
%By the same argument, we can show~(\ref{eq:1}) in the case of $a=t$.
%%%%%%%%
%%
%%
 Then we have $\ds{\eta_x(ca),\gamma_v(ca)}\leq D+1$.
 This implies
 \begin{align*}
%  \label{eq:prod-eta-gamm}
  (\eta_x \mid \gamma_v) \geq ca
  \geq c(t^\frac{1}{\epsilon} - \tilde{\theta}(0)).
 \end{align*}
 First we supposed that $a\geq \tilde{\theta}(0)+1$. 
 Then by an elementary calculus, 
 \begin{align*}
 %\label{eq:3}
  \abs{t - t_v^\epsilon} \leq 
  \epsilon\abs{t^{\frac{1}{\epsilon}}-t_v}
  \leq \epsilon\tilde{\theta}(0).
 \end{align*}
 We remark that $t^{\frac{1}{\epsilon}}-\tilde{\theta(0)} \geq 1$ 
 since $t^{\frac{1}{\epsilon}}\geq a$. Then we have
  \begin{align}
   d_{\OdoX}(t_v^\epsilon[\gamma_v], t[\eta_x]) 
   &\leq \abs{t-t_v^\epsilon} 
    + \min\{t_v^\epsilon,t\}d_\epsilon([\gamma_v],[\eta_x]) 
     \notag\\  
   &\leq \epsilon\tilde{\theta}(0)
     + t\qds([\gamma_v], [\eta_x])
     \notag\\  
   &\leq \epsilon\tilde{\theta}(0)
     + t\left(t^{\frac{1}{\epsilon}}-
     \tilde{\theta}(0)\right)^{-\epsilon}c^{-\epsilon}.
  \label{eq:4}
  \end{align}
 the second term in
 (\ref{eq:4}) is bounded from above by a universal constant.  

 Next we suppose that $a< \tilde{\theta}(0)+1$. Then we have 
 \begin{align*}
  d_{\OdoX}(t_v^\epsilon[\gamma_v], t[\eta_x])\leq t_v^\epsilon 
  + t<2(2\tilde{\theta}(0)+1)^\epsilon.
 \end{align*}
 These show
 that $\loge\circ\expe$ is close to the identity.
\end{proof}

\begin{lemma}
% Set $Y:= \expe(\OdoX)\subset X$.
 The composite $\expe\circ \radcon\circ \loge$ 
 is coarsely homotopic to the identity on $\expe(\OdoX)$.
\end{lemma}

\begin{proof}
Since $\radcon\colon \OdoX \to \OdoX$ is coarsely homotopic to the
identity, the map $\expe\circ\radcon\circ\loge$ is coarsely
homotopic to the map $\expe\circ\loge$. Thus it is enough to
show that $\expe\circ\loge$ is close to the identity.

For $v\in \expe(\OdoX)$, 
let $\gamma_v\in\calL_O^\infty$ and $t_v\in\Rp$ are the
quasi-geodesic ray and the parameter, respectively, chosen in the definition
of the logarithmic map. Thus $\gamma_v(t_v)=v$.

Set $x=[\gamma_v]$. Let $\eta_x\in\calL_O^\infty$ 
be the quasi-geodesic ray representative of $x$
chosen in the definition of the exponential map. 
Now we have 
$\expe\circ\loge(v) = \expe(t_v^\epsilon[\eta_x]) = \eta_x(t_v)$.
Since $[\gamma_v] =x=[\eta_x]$, by Lemma \ref{lem:HdistD},
we have $\ds{\gamma_v(t_v),\eta_x(t_v)}\leq D$. Thus
$\ds{v,\eta_x(t_v)}\leq D$.
This shows that $\exp\circ\loge$ is close to the identity.
\end{proof}

Summarizing the argument above, we obtain the following.
\begin{proposition}
\label{prop:chmtp-cone-X}
 The map 
$\expe\circ\radcon\colon \OdoX \to \expe(\OdoX)$
 is a coarse homotopy equivalence map.
\end{proposition}

\subsection{Coarse homotopy between 
\texorpdfstring{$\expe(\OdoX)$}{expe(OdoX)} 
and \texorpdfstring{$X$}{X}}
\label{sec:c-homot-betw-exp-X}
%We define a map $\varphi\colon X \to B_{1}(\exp(\OdoX))$ as follows.
Set $D_5:= 2D_1+2k_1$, $D_6:=ED_5+D$ 
and $Y:= B_{D_6}(\expe(\OdoX))$. 
There exists a subset $X^{(0)}\subset X$ such that
$X^{(0)}$ is 2-dense in $X$ and 1-discrete, that is,
for all $v,w\in X^{(0)}$, if $v\neq w$
then $\ds{v,w}\geq 1$, 
and, for all $v\in X$, there exists $v'\in X^{(0)}$ with $\ds{v,v'}\leq 2$.
We can assume that $X^{(0)}\cap Y$ is 2-dense in $Y$.
We fix a map $\iota\colon X\to X$ such that
$\iota(v)\in X^{(0)}$ and $\ds{\iota(v),v}\leq 2$ for all $v\in X$, and 
$\iota(v)\in X^{(0)}\cap Y$ for all $v\in Y$.
The purpose of this subsection is to prove the following.
\begin{proposition}
\label{prop:image-exp}
The inclusion $Y\hookrightarrow X$
is a coarse homotopy equivalence map.
% The image of the exponential map $B_{D_5}(\exp(\OdoX))$ is
% coarsely homotopy equivalent to $X$.
\end{proposition}

For $v\in X^{(0)}$, we choose a quasi-geodesic segment 
$\gamma_v\in \calL_O$ and a parameter $T_v\in \Rp$
such that $\gamma_v(0) = O$ and $\gamma_v(T_v)=v$.
Set $s_v:= \sup\{t: \ds{\gamma_v(t),\expe(\OdoX)}\leq D_5\}$.
%Set $s_v:= \max\{t: \gamma_v(t)\in B_{D_5}(\expe(\OdoX))\}$.
We remark that $s_v\geq 0$ since $O=\gamma_v(0)\in \expe(\OdoX)$.

\begin{lemma}
\label{lemma:properness}
 For each $N\geq 0$, the cardinality of the set 
 $\{v\in X^{(0)}:s_v\leq N\}$ is finite.
\end{lemma}
\begin{proof}
% It is easy to see that $\varphi$ is bornologous. Thus we will show that
% it is proper. In fact, we will show that for each vertex $v\in
% \expe(\OdoX)$, the cardinarity $\sharp \varphi^{-1}(v)$ is
% finite.
% Since the closed ball $B_R(O)$ is compact, it is enough to show that
% for each $p\in X$, the cardinality of the set 
%$\{v:\gamma_v(s_v)\in B_{D_5}(p)\}$ is finite.

 We suppose that $\{v\in X^{(0)}:s_v\leq N\}= \infty$. 
 Since $\bar{X}=X\cup\partial_O X$ is compact and $X^{(0)}$ is uniformly
 discrete, we can choose a sequence $v_i\in \{v\in X^{(0)}:s_v\leq N\}$
 which converges to a point $x\in \partial_O X$.  
 We choose a quasi-geodesic ray $\eta\in \calL_O^\infty$ 
 such that $x = [\eta]$. 

 For sufficiently large $n$, we have $(v_n,x) \in V_{D_3N}$, 
 where $V_{D_3N}$ is an entourage of the uniform structure
 defined in Section~\ref{sec:topon-X-dX}.
Let $\gamma_{v_n}\in \calL_O$ 
be a quasi-geodesic segment for $v_n$
chosen in the beginning of Section~\ref{sec:c-homot-betw-exp-X}.
% 
% Then by Lemma~\ref{lem:X-dX-chooseRep}, we have 
 Set $a:=(\gamma_{v_n}\mid \eta)$.
 We have $\ds{\gamma_{v_n}(a), \eta(a)}\leq D_1+2k_1\leq D_5$. It follows that
 \begin{align*}
  s_{v_n}\geq a =(\gamma_{v_n}\mid \eta) \geq D_3^{-1}(v_n\mid x)>N.
 \end{align*}
 This contradicts that $v_n\in \{v\in X^{(0)}: s_v\leq N\}$.
\end{proof}

%Now we are ready to define a map 
%$\varphi\colon X^{(0)} \to Y$. 

For each positive integer $n\in \N$, we define a sequences $l(n)$ by
\begin{align*}
 l(n):= \max\{T_v: v\in X^{(0)}, \, n\leq s_v<n+1\}.
\end{align*}
By Lemma~\ref{lemma:properness}, each $l(n)$ is finite. 
%The sequence $l(n)$ is not necessarily increasing, 
We choose a subsequence
$n_i$ satisfying the following
\begin{align*}
 &l(n_1) > 1,\\
 &l(n_{i+1}) - l(n_{i})> 1, \quad (i\geq 1),\\
 &l(n_i)>l(n),   \quad (i\geq 1, \, 1\leq n< n_i).
\end{align*}

\begin{lemma}
\label{lem:Tv-ni}
 For $v\in X^{(0)}$ and $i\geq 1$, 
 if $l(n_i)\leq T_v$, then we have $n_i\leq s_v$.
\end{lemma}
\begin{proof} 
 For $v\in X^{(0)}$, let $i$ be an integer such that 
 $l(n_i)\leq T_v$.
 For $n\in \N$ with $n\leq s_v< n+1$, we have $T_v\leq l(n)$.
 We suppose that $n<n_i$. Then we have $l(n)<l(n_i)\leq T_v$.
 This is a contradiction. Thus we have $n_i\leq s_v$.
\end{proof}

We define a map $\chi\colon \Rp \to \Rp$ by
\begin{align*}
 \chi(t)= 
 \begin{cases}
  0 &0\leq t<l(n_1),\\
  i & l(n_i) \leq t< l(n_{i+1}),\, i\geq 1. 
 \end{cases}
\end{align*}
%$\chi(t):=0$ for $0\leq t<l(n_1)$ and
%$\chi(t):= i$ for $l(n_i) \leq t< l(n_{i+1})$ with $i\geq 1$. 
Then $\chi$ satisfies $\chi(t)\leq t$ and
$\abs{\chi(t)-\chi(s)}\leq \abs{t-s}+1$ for $t,s\in \Rp$.
We define a map $\varphi\colon X^{(0)} \to Y$ by
\begin{align*}
 \varphi(v):= \gamma_v(\chi(T_v)) \quad (v\in X^{(0)}).
\end{align*}
Since $\chi(T_v) = i< n_i$, by Lemma~\ref{lem:Tv-ni}, we have 
$\chi(T)< s_v$. It follows that $\varphi(v)\in Y$.

\begin{lemma}
 The map $\varphi$ is a coarse map.
\end{lemma}
\begin{proof}
 First we show that $\varphi$ is proper.
 We fix $R>0$. Let $v\in X^{(0)}$ be a point with 
 $\ds{\varphi(v),O}\leq R$. Since
 $(1/\lambda)\chi(T_v)-k\leq \ds{\gamma_v(\chi(T_v)),O}\leq R$, we have
 $\chi(T_v)\leq \lambda(R+k)$. Let $j\in \N$ be an integer with 
 $j> \lambda(R+k)$. Then $T_v< l(n_{j})$, so 
 $\ds{v,O}< \lambda l(n_{j})+k$. This shows that $\varphi$ is proper.
 
 Now we show that $\varphi$ is bornologous. 
 For $v,w \in X^{(0)}$, set $i:= \chi(T_v)$ and $j:= \chi(T_w)$. 
 Then
  \begin{align*}
   \abs{i-j} = \abs{\chi(T_w)-\chi(T_v)}
   \leq \abs{T_w - T_v} +1 \leq \theta(\ds{v,w})+1.
  \end{align*}
 By Lemma~\ref{lem:ray-same-param},
 \begin{align*}
  \ds{\gamma_v(i),\gamma_w(i)}
 \leq E(\ds{v,w} + \lambda\tilde{\theta}(\ds{v,w})+k_1)+ D.
 \end{align*}
 Since $\gamma_w$ is a $(\lambda,k)$-quasi-geodesic segment,
\begin{align*}
 \ds{\gamma_w(i),\gamma_w(j)}\leq \lambda\abs{i-j} + k
  \leq \lambda(\theta(\ds{v,w})+1) +k.
\end{align*}
  Then we have 
 \begin{align*}
  \ds{\varphi(v),\varphi(w)} = \ds{\gamma_v(i),\gamma_w(j)}
  \leq E(\ds{v,w} + \lambda\tilde{\theta}(\ds{v,w})+k_1)+ D 
  + \lambda(\theta(\ds{v,w})+1) +k.
 \end{align*}
  Therefore $\varphi$ is bornologous.
\end{proof}
Set $\tilde{\varphi}:=\varphi\circ \iota\colon X\to Y$.
Let $i\colon Y \hookrightarrow X$ 
be the inclusion. We will show that $i\circ \tilde{\varphi}$ and 
$\tilde{\varphi} \circ i$ are,
respectively, coarsely homotopic to the identity $\mathrm{id}_X$ and 
$\mathrm{id}_Y$.

Indeed, since $\iota$ is close to the identity, 
it is enough to show that $i\circ \tilde{\varphi}$ and 
$\tilde{\varphi} \circ i$ are,
respectively, coarsely homotopic to the map $\iota$ and the restriction
$\iota|_Y$ of $\iota$ on $Y$.
First we construct a coarse homotopy between $i\circ \tilde{\varphi}$ and 
$\iota$.

Set $Z:= \{(v,t)\in X\times \R: 0\leq t\leq T_{\iota(v)}\}$. 
We remark that the map $X\ni v\mapsto T_{\iota(v)}\in \R_{\geq 0}$ is 
bornologous.
We define a map $H\colon Z\to X$ by 
\begin{align*}
 H(v,t):= \gamma_{\iota(v)}(T_{\iota(v)} -t +\chi(t)).
\end{align*}
It is easy to see that $H(v,0)=\iota(v)$ 
and $H(v,T_{\iota(v)}) = i \circ \tilde{\varphi}(v)$.

%\begin{lemma}
% The map $H_\infty\colon X\to X$ defined by 
% $H_\infty(v):=H(v,\ds{O,v})$ is close to the identity.
%\end{lemma}
%\begin{proof}
% For $v\in X$, we have $(1/\lambda)t_v -k \leq \ds{O,v}\leq \lambda t_v+k$, 
%thus $\abs{\ds{O,v}-t_v}\leq 2k$. Therefore, 
%we have $\abs{H_\infty(v)\, v} \leq  
% D_5 + \ds{\gamma_v(\ds{O,v})\, \gamma_v(t_v)}\leq 2k\lambda + k + D_5$.
%\end{proof}

%\begin{lemma}
%\label{lem:ceq-to-geodsp}
% For $(v,t),(w,s)\in Z$ with $\lceil \ds{v,w}\rceil = n-1$, 
% there exists a sequence 
% $(v,t)=(v_0,t_0),(v_1,t_1),\dots,(v_n,t_n)=(w,s)\in Z$ such that
% \begin{align*}
%  \sum_{i=1}^{n}d_Z((v_{i-1},t_{i-1}),(v_i,t_i))
%  \leq 3d_Z((v,t),(w,s)).
  %=3(n+\abs{t-s}).
% \end{align*}
%\end{lemma}

%\begin{proof}
%Since $X$ is a quasi-geodesic space, for $(v,t),(w,s)\in Z$ with 
% $\lceil\ds{v,w}\rceil=n$,
% there exists a sequence $v=v_0,v_1,\dots,v_n=w,\in X$ such that
% $\sum_{i=1}^{n}\abs{v_{i-1}v_i} = n$ and $\abs{v_{i-1}v_i}\leq 1$ for
% all $1\leq i\leq n$. We can show inductively that 
 %% 
% $\ds{O,v_i}\geq \ds{O,v} -i$ for all $0\leq i\leq n$.  
 %%
% Set $t_i:= t-i$ for $0\leq i\leq n-1$ and $t_n:=s$. Then $(v_i,t_i)\in Z$
% for all $0\leq i\leq n$.
% We have
% \begin{align*}
%  \sum_{i=1}^{n}d_Z((v_{i-1},t_{i-1}),(v_i,t_i)) &=
%   \sum_{i=1}^{n}(\abs{v_{i-1}v_i} + \abs{t_i-t_{i-1}})\\
%  &\leq n+ n-1 + \abs{s-(t-n+1)}\\ %% comment: we can skip this line.
%  &\leq 3(n+\abs{t-s}).
% \end{align*}
%\end{proof}

\begin{lemma}
\label{lem:H-coarse}
 The map $H$ is a coarse map.
\end{lemma}

\begin{proof}
 It is easy to show that $H$ is proper. Thus we show that it is
 bornologous.  We fix $(v,t),(w,s)\in Z$. We remark that
 $t\geq \chi(t)$ and $s\geq \chi(s)$. Set
 $v':=\iota(v)$ and $w':=\iota(w)$.  We suppose $T_{w'}\geq
 T_{v'}$. Then
\begin{align*}
 \ds{H(v,t), H(w,s)}=&
 \ds{\gamma_{v'}(T_{v'} -t +\chi(t)), 
     \gamma_{w'}(T_{w'} -s +\chi(s))} \\
 \leq&  
 \ds{\gamma_{v'}(T_{v'} -t +\chi(t)), 
     \gamma_{w'}(T_{v'} -t +\chi(t))} \\
 &+
 \ds{\gamma_{w'}(T_{v'} -t +\chi(t)), 
     \gamma_{w'}(T_{w'} -s +\chi(s))} \\
 \leq& E\ds{\gamma_{v'}(T_{v'}), 
            \gamma_{w'}(T_{v'})} +C \\
 &+\lambda\abs{T_{v'} -t +\chi(t) -(T_{w'} -s +\chi(s))}+k\\
 \leq & E(E(\ds{v',w'}+\lambda\tilde{\theta}(\ds{v',w'})+k_1)+D)+C\\
 &+\lambda(\theta(\ds{v',w'})+2\abs{t-s}+1) + k.
\end{align*}
Since $\ds{v',w'}\leq \ds{v,w}+4$, it follows that $H$ is bornologous.
\end{proof}

\begin{corollary}
 The map $i\circ \tilde{\varphi}$ and $\mathrm{id}_X$ are coarsely homotopic.
\end{corollary}

Now we construct a coarse
homotopy between $\tilde{\varphi}\circ i$ and $\iota|_{Y}$.  
Set $Z':=Z\cap (Y\times \Rp)$. Let $H'$ be the restriction of $H$ to
$Z'$. Then the range of $H'$ is in $Y$. 
It follows that $H'$ is a coarse map
and $H'(v,0) = \iota(v)$, $H'(v,T_{\iota(v)}) = \tilde{\varphi}\circ i(v)$.
\begin{corollary}
 The map $\tilde{\varphi} \circ i$ and $\mathrm{id}_Y$ are 
 coarsely homotopic.
\end{corollary}

This completes the proof of Proposition~\ref{prop:image-exp}.  Combining it
with Proposition~\ref{prop:chmtp-cone-X}, we obtain
Theorem~\ref{thm:cCATHAD}.

\section{Application to the Coarse Baum-Connes conjecture}
\label{sec:appl-coarse-baum}
\subsection{Coarse Baum-Connes conjecture}
\label{sec:coarse-baum-connes}
The coarse category is a category whose objects are proper metric spaces
and whose morphisms are close classes of coarse maps.  Let $X$ be a
proper metric space. There are two covariant functors $X \mapsto
KX_*(X)$ and $X \mapsto K_*(C^*(X))$ from the coarse category to the
category of $\Z_2$-graded Abelian groups. Here the $\Z_2$-graded Abelian
group $KX_*(X)$ is called the {\itshape coarse $K$-homology} of $X$, and the
$C^*$-algebra $C^*(X)$ is called the {\itshape Roe algebra} of $X$.  
Roe \cite{MR1147350} constructed the following {\itshape coarse assembly map}
\begin{align*}
 \mu_*&\colon KX_*(X) \to K_*(C^*(X)),
\end{align*}
which is a natural transformation from the coarse
$K$-homology to the $K$-theory of the Roe algebra. 
For detail, see also \cite{MR1388312}, \cite{MR1344138} and
\cite{MR1817560}.

The important feature of these functors is, both  
the coarse K-homology and the K-theory of the Roe-algebra
are coarse homotopy invariants in the following sense. 

\begin{proposition}
 Let $X$ and $Y$ be proper metric spaces.  If there exists a coarse
 homotopy equivalence map $f\colon X\to Y$, then in the
 following commutative diagram, two vertical homomorphisms both denoted
 by $f_*$ are isomorphisms
 \begin{align*}
  \xymatrix{
  KX_*(X) \ar[d]^\cong_{f_*} \ar[r]^{\mu_*} 
     &K_*(C^*(X))\ar[d]^\cong_{f_*}  \phantom{.}\\
  KX_*(Y) \ar[r]^{\mu_*} &  K_*(C^*(Y)).
  }
 \end{align*}
\end{proposition}
Coarse homotopy invariance is proved by {\itshape Mayer-Vietoris
principle}. For details, see \cite[Proposition 12.4.12]{MR1817560} and
\cite[Theorem 4.3.12]{WillettThesis}.

\begin{corollary}
\label{cor:hmtpeq-cBC} Let $X$ and $Y$ are proper metric
 spaces. We suppose that $X$ and $Y$ are coarsely homotopy
 equivalent.  If the coarse assembly map $\mu_*\colon KX_*(Y)\to
 K_*(C^*(Y))$ is an isomorphism, then so is the coarse assembly map
 $\mu_*\colon KX_*(X)\to K_*(C^*(X))$.
\end{corollary}

Let $M$ be a compact metric space. 
Higson-Roe~\cite[Section 7]{MR1388312} showed that 
the coarse Baum-Connes conjecture holds for the open cone $\calO M$. We
remark that in~\cite{MR1388312}, $M$ is assumed to be finite
dimensional.  However, by \cite[Appendix B]{relhypgrp}, we can remove
this assumption.
\begin{theorem}%[Higson-Roe]
\label{thm:open-cone-cBC}
 Let $M$ be a compact metric space. Then the coarse assembly map
 \begin{align*}
 \mu_*&\colon KX_*(\calO M) \to K_*(C^*(\calO M))
\end{align*}
 is an isomorphism.
\end{theorem}

\begin{proof}
 [Proof of Theorem~\ref{thm:cconvex-cBC}] Let $X$ be a proper
 coarsely convex space.  By Theorem~\ref{thm:cCATHAD}, $X$ is
 coarsely homotopy equivalent to the open cone $\OdoX$.
 Then by Theorem~\ref{thm:open-cone-cBC} and
 Corollary~\ref{cor:hmtpeq-cBC}, the coarse assembly map
 $\mu_*\colon KX_*(X)\to K_*(C^*(X))$ is an isomorphism.
\end{proof}
\subsection{Coarse compactification}
\label{sec:coarse-comp}
Let $X$ be a non-compact proper metric space.  Let $\varphi\colon X\to
\C$ is a function. We say that $\varphi$ is {\itshape slowly oscillating}
if for any $\epsilon>0$ and $R>0$, there exists a bounded subset
$B\subset X$ such that
 \begin{align*}
  \sup\{\abs{\varphi(v)-\varphi(w)}
       : v,w\in X\setminus B,\, \ds{v,w}\leq R\}< \epsilon.
 \end{align*}

%The {\itshape Higson compactification} $hX$ of $X$ is the maximal ideal
%space of the $C^*$-algebra of $\C$-valued, continuous, bounded, slowly
%oscillating functions on $X$ . We say that a compactification $\bar{X}$
%of $X$ is a {\itshape coarse compactification} if the identity $X\to X$
%extends to a continuous map $hX \to \bar{X}$.
%%
%The following is a criterion for coarse compactifications.
\begin{definition}
 Let $X$ be a proper metric space, and let $\bar{X}$ be a
 compactification of $X$. Then $\bar{X}$ is a 
 {\itshape coarse compactification} if
 for any continuous map $\varphi \colon \bar{X}\to \C$, 
 the restriction of $\varphi$ to $X$ is slowly oscillating.
\end{definition}
For detail on coarse compactifications, see \cite[Section 2.2]{MR2007488}, 
\cite[Section 5.1]{MR1147350} or \cite[Section 2.2]{boundary}.
Let $\bar{X}$ be a coarse compactification of $X$. 
Set $\partial X:=\bar{X} \setminus X$. Then there exists a
certain transgression map
\begin{align}
\label{Tr:1} &{T_{\partial X}} \colon KX_*(X) 
 \to \tilde{K}_{*-1}(\partial X).
\end{align}
Here $\tilde{K}_*(\partial X)$ is the reduced K-homology of $\partial X$.
Higson-Roe constructed a homomorphism 
$b\colon K_*(C^*(X))\to \tilde{K}_{*-1}(\partial X)$
such that $T_{\partial X} = b\circ \mu_*$.
Therefore if the transgression
map (\ref{Tr:1}) is injective, then so is the coarse assembly map for $X$.
See \cite[9. Appendix]{MR1388312} and \cite{boundary} for detail.

Let $M$ be a compact metrizable space. The open cone $\calO M$ has a
natural compactification $\calO M \cup M$ by attaching $M$ at
infinity. Indeed, we set $\mathcal{C}M:=[0,1]\times M/(\{0\}\times M)$
and define $\varphi\colon \Rp \to [0,1)$ by
$\varphi(t):=t/(1+t)$. Then a map $tx\mapsto \varphi(t)x$ gives an
embedding of $\calO M$ into $\mathcal{C}M$ with an open dense image.

\begin{proposition}
 Let $M$ be a compact metric space. Then the compactification 
 $\calO M \cup M$ is a coarse compactification, and
 the transgression map
 \begin{align*}
  T_{M} \colon KX_*(\calO M) \to \tilde{K}_{*-1}(M)
 \end{align*}
 is an isomorphism.
\end{proposition}
For the proof see 
\cite[Proposition 4.3]{MR1388312}, \cite[Lemma 4.5.3]{WillettThesis} 
or \cite[Lemma 5.1]{boundary}.

\begin{proposition}
\label{prop:vanishing-var}
 Let $X$ be a proper coarsely convex space. 
 Then $\bar{X}=X\cup \partial X$ is a coarse compactification, 
 where $\partial X$ is the ideal boundary of $X$.
\end{proposition}

\begin{proof}
 Let $X$ be a $(\lambda,k,E,C,\theta,\calL)$-coarsely convex
 space.  
 Let $O\in X$ be the base point and 
 let $\partial X$ be the ideal boundary with respect to $O$. 
%We will show that
% $\bar{X}=X\cup \partial X$ satisfies the assumption in
% Proposition~\ref{prop:vanishing-var}.
 
 Let $\varphi \colon \bar{X}\to \C$ be a continuous map. We will show
 that $\varphi$ is slowly oscillating. Since $X\cup \partial X_O$ is
 compact, for any $\epsilon>0$, there exists $n>0$ such that if
 $(p,q)\in V_n$ then
 $\abs{\varphi(p)-\varphi(q)}<\epsilon $, where $V_n$ is an entourage of
 the uniform structure defined in Section~\ref{sec:topon-X-dX}. 
 
 Now we fix $R>1$ and set 
% $d:= \lambda n\{E(E(R + \lambda\tilde{\theta}(R)+k_1)+D)\}+k$.
 $d:= \lambda n\{E(R + \lambda\theta(R)+k)\}+k$.
 Let $v,w\in X\setminus B_{d}(O)$ be points with $\ds{v,w}\leq R$.
 Let $\gamma_v,\gamma_w\in \calL_O$ be quasi-geodesic segments such that
 $\gamma_v\colon [0,a_v]\to X$, $\gamma_w\colon [0,a_w]\to X$, 
 $\gamma_v(a_v) =v$ and $\gamma_w(a_w)=w$. Set $a:=\min\{a_v,a_w\}$. 
 We remark that $a\geq (d-k)/\lambda\geq n$. 
 We can suppose without loss of generality that $a=a_v\leq a_w$.
Then 
 \begin{align*}
  \ds{\gamma_v(n),\gamma_w(n)}\leq& 
  \frac{n}{a}E\ds{\gamma_v(a),\gamma_w(a)} +C\\
  \leq& \frac{\lambda n}{d-k}E
  (\ds{\gamma_v(a_v),\gamma_w(a_w)}+\ds{\gamma_w(a_w),\gamma_w(a_v)}) +C\\
  \leq& \frac{\lambda n}{d-k}E(R + \lambda\theta(R)+k)+C\\
  \leq& 1+C \leq D_1.
 \end{align*}
 It implies that $(v\mid w)\geq (\gamma_v\mid \gamma_w) \geq n$. 
 Thus $(v,w)\in V_n$, and so $\abs{\varphi(v)-\varphi(w)}<\epsilon$.
\end{proof}

\begin{theorem}
 Let $X$ be a proper coarsely convex space. 
 Then the transgression map
 \begin{align*}
  {T_{\partial X}} \colon KX_*(X) \to \tilde{K}_{*-1}(\partial X).  
 \end{align*}
 is an isomorphism.
\end{theorem}
\begin{proof}
 The statement follows immediately from the following diagram.
 \begin{align*}
  \xymatrix{
  KX_*(X) \ar[d]^\cong \ar[r]^{T_{\partial X}} 
     &\tilde{K}_{*-1}(\partial X) \ar@{=}[d]  \\
  KX_*(\OdoX) \ar[r]^{T_{\partial X}} &\tilde{K}_{*-1}(\partial X)
  }
 \end{align*}
\end{proof}

%%%%%%%%%%%% direct product with polycyclic group %%%%%%%%%%%%%%
\subsection{Direct product with polycyclic groups}

One of advantages of the coarse Baum-Connes conjecture is that
the coarse Mayer-Vietoris principle holds for both sides of the
coarse assembly maps.  As an application of this, we have the
following (\cite[Proposition 7.2]{product-cBC}).

\begin{proposition}
\label{prop:assembly-map-solv}
Let $\LieG$ be a simply connected solvable Lie group 
with a lattice.
We equip $\LieG$ with a proper left invariant metric.
Let $Y$ be a proper metric space.
Suppose that $Y$ satisfies the coarse Baum-Connes conjecture.
Then so does the direct product $Y\times \LieG$.
\end{proposition}

\begin{corollary}
\label{cor:cconvexXpolyc}
 Let $\LieG$ be a simply connected solvable Lie group with a lattice,
 and let $X$ be a proper coarsely convex space.
 Then the direct product $X\times \LieG$ satisfies the coarse
 Baum-Connes conjecture.
\end{corollary}

We remark that every polycyclic group $G$ admits a normal subgroup $G'$
of finite index in $G$ which is isomorphic to a lattice in a simply
connected solvable Lie group. See \cite[Theorem 4.28]{MR0507234}.

%Corollary~\ref{cor:prod-with-polyc} follows from
%Corollary~\ref{cor:cconvexXpolyc} and the above remark.

%Corollary~\ref{cor:syst-cBC} and Proposition~\ref{prop:assembly-map-solv}.

%%%%%%%%%%%% relatively hyperbolic %%%%%%%%%%%%%%

\subsection{Relatively hyperbolic groups}

In \cite{relhypgrp}, the authors studied
the coarse Baum-Connes conjecture for relatively hyperbolic groups.
\begin{theorem}
[\cite{relhypgrp}]
\label{relhyp_12} 
 Let $G$ be a finitely generated group and
 $\famP=\{P_1,\dots,P_k\}$ be a finite family of subgroups.  Suppose
 that G is hyperbolic relative to $\famP$.  If each subgroup $P_i$
 satisfies the coarse Baum-Connes conjecture, and admits a finite
 $P_i$-simplicial complex which is a universal space for proper actions,
 then $G$ satisfies the coarse Baum-Connes conjecture.  Moreover, if
 $G$ is torsion-free and each subgroup $P_i$ is classified by a
 finite simplicial complex, then $G$ satisfies the Novikov
 conjecture.
\end{theorem}

%We remark that injectivity of the coarse assembly maps for peripheral
%subgroups is not enough to prove the Novikov conjecture for relative
%hyperbolic groups using the argument given in \cite{relhypgrp}.

%Let $G$ be a group acting geometrically on a locally finite systolic
%complex $X$ by simplicial automorphisms. It is shown in
%\cite{Chepoi-Osajda-Dismant2015} that the barycentric subdivision of $X$
%is a universal space for proper action. 
%If $G$ is torsion free, then the
%quotient $X/G$ is a finite classifying space.

Let $\classNPCPC$ be a class of groups consisting of all finite direct
products of hyperbolic groups, CAT(0)-groups, systolic groups, and
polycyclic groups.  Each group $P$ in $\classNPCPC$ admits a finite
$P$-simplicial complex which is a universal space for proper actions.
We refer \cite{Chepoi-Osajda-Dismant2015} for the case of
systolic groups.  If $P$ in $\classNPCPC$ is torsion free,
then $P$ is classified by a finite simplicial complex.
Now 
Theorem~\ref{thm:relhypPC} follows from Theorem~\ref{thm:cconvex-cBC}, 
Theorem~\ref{relhyp_12} and
Proposition~\ref{prop:assembly-map-solv}, 

\begin{remark}
 The $3$-dimensional discrete Heisenberg group never act geometrically
 on any coarsely convex space, since it does not satisfy any quadratic
 isoperimetric inequality~\cite[Example 8.1.1]{EetalWPG}, which violate
 the conclusion of Corollary~\ref{cor:CCGroup}.  Hence
 Theorem~\ref{thm:relhypPC} does not follows directly from
 Theorem~\ref{thm:cconvex-cBC}.
\end{remark}

%since the free product can be regarded as a relatively hyperbolic group. 
%%
More generally, by a similar argument as the proof of
\cite[Theorem 1.1]{product-cBC}, we can show the following.

\begin{theorem}
\label{thm:prod-relhyp} 
Let $m$ be a positive integer. 
For $1\leq j\leq m$, let $G^j$ be a group in $\classNPCPC$, or, be a
finitely generated group which is
hyperbolic relative to a family of subgroups 
$\famP^j= \{P^j_1,\dots,P^j_{k^j}\}$ consisting of members 
of $\classNPCPC$. Then the direct product group
%$\prod_{j=1}^m G^j \times \prod_{i=1}^l H_i$ 
\begin{align*}
 G:= G^1\times \dots \times G^m
\end{align*}
satisfies the coarse Baum-Connes conjecture.  Moreover, if $G$
is torsion-free, then $G$ satisfies the Novikov conjecture.
\end{theorem}

%%%semihyperbolicity%%%%%%%%%%%%%%%%%%%%%%%%%
%\section{Semihyperbolicity of coarsely convex spaces}
\section{Groups acting on a coarsely convex space}
\label{sec:groups-acting-coars}
From the view point of geometric group theory, it is natural to consider
groups acting on coarsely convex spaces. In this section, we mention
some algebraic properties of such groups, which follows immediately from
{\itshape semihyperbolicity} of coarsely convex spaces.

%In Section~\ref{sec:appl-coarse-baum}, we have already discussed on
%groups acting on a coarsely convex spaces from the view point of the
%Novikov conjectute.

%\subsection{Semihyperbolicity}
%\label{sec:semihyp-spac} 

Alonso and Bridson~\cite{A-BSemihyp} introduce another formulation of
``nonpositively curved space'', called semihyperbolic space.
We show that a coarsely convex space is semihyperbolic in their sense.

%It follows imediately that some algebraic properties of groups
%acting on a coarsely convex space.

%we give a relation to semihyperbolic spaces in the
%sense of Alonso and Bridson~\cite{A-BSemihyp}. 

First, we briefly review the definition and properties.
Let $X$ be a metric space.  A {\itshape discrete path} is a map
$\gamma\colon [0,T_\gamma]\cap \Z \to X$ with 
$T_\gamma\in \N\cup \{0\}$.
For convenience, we consider $\gamma$ as a map 
$\gamma\colon \N\cup \{0\}\to X$ by
setting $\gamma(t):=\gamma(T_\gamma)$ if $t\geq T_\gamma$. 
Let $\mathscr{P}'(X)$ be the set of discrete paths.
We consider the endpoints map $e\colon \mathscr{P}'(X) \to X\times X$ given 
by $e(\gamma)=(\gamma(0),\gamma(T_\gamma))$.

A {\itshape bicombing} is a section $s\colon X\times X \to
\mathscr{P}'(X)$ of the endpoints map $e$. We denote the image of
$(x,y)$ by $s_{(x,y)}$.

A bicombing $s$ is said to be {\itshape quasi-geodesic} if there
exist constants $\lambda, k$ such that $s_{(x,y)}$ is a
$(\lambda, k)$-quasi-geodesic segment for all $x,y \in X$.

A bicombing $s$ is called {\itshape bounded} if there exist
constants $k_1\geq 1$, $k_2\geq 0$ such that, for all $x,y,x',y'\in X$
and $t\in \N\cup \{0\}$, 
\begin{align*}
 \ds{s_{(x,y)}(t),s_{(x',y')}(t)} \leq k_1\max\{\ds{x,x'},\ds{y,y'}\} +k_2.
\end{align*}

\begin{definition}[\cite{A-BSemihyp}]
 A metric space $X$ is {\itshape semihyperbolic} if it admits a bounded
 quasi-geodesic bicombing.
\end{definition}

Alonso and Bridson~\cite[Theorem 1.1]{A-BSemihyp} showed 
that being semihyperbolic is invariant
under quasi-isometries. Then they studied groups acting on a 
semihyperbolic space.

\begin{theorem}[{\cite[Theorem 2.8 and Theorem 5.1]{A-BSemihyp}}]
 \label{thm:semihypGrp}
 Let $G$ be a group acting on a semihyperbolic space $X$ properly and
 cocompactly by isometries. Then the following holds.
 \begin{enumerate}
%  \item \label{item:FP} $G$ is finitely presented.
  \item $G$ is finitely presented and of type $FP_{\infty}$.
  \item $G$ satisfies a quadratic isoperimetric inequality.
 \end{enumerate}
 Moreover, suppose that a bicombing $s$ of $X$ is $G$-invariant, then
 \begin{enumerate}
 \setcounter{enumi}{2}
  \item $G$ has a solvable conjugacy problem.  
  \item Every polycyclic subgroup of $G$ contains
	a finitely generated abelian subgroup of finite index.
 \end{enumerate}
\end{theorem}
%For the proof of Theorem~\ref{thm:semihypGrp}, see Section 2 and Section
%5 in \cite{A-BSemihyp}.

\begin{proposition}
 \label{prop:semihyperbolicity} Let $X$ be a coarsely convex space. Then
 $X$ is semihyperbolic.  Moreover, suppose that a group $G$ acts on $X$
 by isometries, and $G$ preserves a system of good quasi-geodesic
 segments $\calL$ of $X$, then $X$ admits $G$-invariant bounded
 quasi-geodesic bicombing.
  %in the sense of Alonso and Bridson.
\end{proposition}

\begin{proof}
 Let $X$ be $(\lambda,k,E,C,\theta,\calL)$-coarsely convex.  Then we can
 assume that $\theta$ is a large scale Lipschitz function.
 Indeed for every $(x,y)\in X\times X$, we take $\gamma_{x,y}\in \calL$
 whose domain is $[0,a_{x,y}]$ and $t_{x,y}\in [0,a_{x,y}]$ with
 $\gamma_{x,y}(0)=x$, $\gamma_{x,y}(t_{x,y})=y$.  Then the map
 \begin{align*}
  X\times X\to \Rp;(x,y)\mapsto t_{x,y}
 \end{align*}
 is $\theta$-bornologous.
 We equip $X\times X$ with the $\ell_1$-metric.
 Since $X$ is quasi-geodesic, so is $X\times X$.
 Hence we have constants $A\geq 1,\, B'\geq 0$ such that
 \begin{align*}
  \abs{t_{x,y}-t_{x',y'}} \leq A(\ds{x,x'}+\ds{y,y'})+B'
 \end{align*}
 for any $(x,y),(x'y')\in X\times X$.
 When for $(x,y)\in X\times X$ we choose different
 $\eta_{x,y}\in \calL$ whose domain is $[0,b_{x,y}]$
 and $u_{x,y}\in [0,b_{x,y}]$ with
 $\eta_{x,y}(0)=x$, $\eta_{x,y}(u_{x,y})=y$,
 we have $\abs{t_{x,y}-u_{x,y}} \leq \theta(0)$.
 We put $B=B'+\theta(0)$. Then we have the following.
 \begin{enumerate}[(i)'$^q$]
  \setcounter{enumi}{2}
  \item \label{qparam-reg'}
  Let $\gamma,\eta\in \calL$ be quasi-geodesic segments
  with $\gamma\colon [0,a]\to X$ and $\eta\colon [0,b]\to X$.
  Then for $t\in [0,a]$ and $s\in [0,b]$, we have
 \begin{align*}
  \abs{t-s} \leq A(\ds{\gamma(0),\eta(0)}+\ds{\gamma(t),\eta(s)})+B.
 \end{align*}
 \end{enumerate}
 Let $\gamma,\eta\in \calL$ be quasi-geodesic segments as in
 (\ref{qparam-reg'})'$^q$.
 Put $\gamma(t)=\gamma(a)$ for any $t\geq a$ and
 $\eta(u)=\eta(b)$ for any $u\geq b$. 
 We assume that $a\leq b$.  Then
% the above argument implies
 (\ref{qparam-reg'})'$^q$ implies
 \begin{align*}
  \ds{\gamma(0),\eta(0)}+\ds{\gamma(a),\eta(a)}&
  \leq \ds{\gamma(0),\eta(0)}
    +\ds{\gamma(a),\eta(b)}+\ds{\eta(a),\eta(b)}\\
  &\leq \ds{\gamma(0),\eta(0)}
    +\ds{\gamma(a),\eta(b)}+\lambda\abs{b-a}+k\\
  &\leq \ds{\gamma(0),\eta(0)}
    +\ds{\gamma(a),\eta(b)}
    +\lambda (A(\ds{\gamma(a),\eta(b)}+\ds{\gamma(0),\eta(0)})+B)+k\\
  &\leq (\lambda A+1)
   (\ds{\gamma(a),\eta(b)}+\ds{\gamma(0),\eta(0)})+\lambda B+k.
 \end{align*} 
 For any $t\leq a$, we have 
 \begin{align*}
  \ds{\gamma(t),\eta(t)}
  &\leq \frac{t}{a}E\ds{\gamma(a),\eta(a)}
  + \frac{a-t}{a}E\ds{\gamma(0), \eta(0)}+ C\\
  &\leq E(\ds{\gamma(a),\eta(a)}
  + \ds{\gamma(0), \eta(0)})+ C\\
  &\leq E(\lambda A+1)(\ds{\gamma(a),\eta(b)}
  + \ds{\gamma(0), \eta(0)})+(E(\lambda B+k)+C).
\end{align*}
 Also for any $t\ge a$,
\begin{align*}
 \ds{\gamma(t),\eta(t)}
 =\ds{\gamma(a),\eta(t)}
 &\leq \ds{\gamma(a),\eta(b)}+\ds{\eta(b),\eta(t)}\\
 &\leq \ds{\gamma(a),\eta(b)}+\lambda\abs{b-t}+k\\
 &\leq \ds{\gamma(a),\eta(b)}+\lambda\abs{b-a}+k\\
 &\leq \ds{\gamma(a),\eta(b)}+\lambda
 (A(\ds{\gamma(a),\eta(b)}+\ds{\gamma(0),\eta(0)})+B)+k\\
 &\leq (\lambda A+1)
 (\ds{\gamma(a),\eta(b)}+\ds{\gamma(0),\eta(0)})+\lambda B+k.
\end{align*}

 Now we can easily construct a bounded quasi-geodesic bicombing from $\calL$.
 The second assertion follows from the construction.
\end{proof}

Corollary~\ref{cor:CCGroup} follows immediately from
Theorem~\ref{thm:semihypGrp} and
Proposition~\ref{prop:semihyperbolicity}. 
Another application is given in Corollary~\ref{cor:cohom-dim}.

An advantage of a group $G$ acting on a coarsely convex space $X$ is, if
$G$ preserve a system of good geodesic segments $\calL$ of $X$,
then $G$ acts on the ideal boundary $\partial X$ of $X$, as we have
already seen in Corollary~\ref{cor:action-on-dX}. We hope more algebraic
and geometric properties of the group $G$ can be understood
through the topology of $\partial X$, such as splitting of $G$, as in
the case of hyperbolic groups by Bowditch~\cite{Bow-cut-pts-hyp} and
that of CAT(0)-groups by Papasoglu-Swenson~\cite{JSJ-CAT0}.

It also seems natural to ask whether the group $G$ admits finite
$G$-simplicial complex which is a universal space for proper actions.

\section{A functional analytic characterization of the ideal boundary}
\label{sec:funct-analyt-char} The aim of this section is to give a
functional analytic characterization of the ideal boundaries of coarsely
convex spaces. As an application, we show that the ideal boundary coincides with
the bicombing corona introduced by Engel and Wulff~\cite{combable-corona}.

  Let $X$ be a proper metric space which is
$(\lambda,k,E,C,\theta,\calL)$-coarsely convex. Let $O$ be a base point.
We denote by $(\cdot \mid \cdot)$ the Gromov product with respect to the
base point $O$. We use constants defined in the beginning of
Section~\ref{sec:ideal-boundary}.

\begin{definition}
\label{def:Gromov-fct}
 We say that a function
 $f\colon X\to \C$ is a {\itshape Gromov} function if for all
 $\epsilon>0$, there exists $R>0$ such that for $v,w\in X$ with $(v\mid
 w)>R$, we have $\abs{f(v)-f(w)}<\epsilon$.
\end{definition}

We denote by $C_g(X)$ a set of continuous Gromov functions. We will show
that the set $C_g(X)$ is in fact an algebra and it is isomorphic to the
algebra of all continuous functions on the ideal boundary
compactification $\bar{X}=X\cup \partial_O X$.

Let $C(X)$ and $C(\bar{X})$ 
be the algebra of continuous complex valued
functions on $X$, and on $\bar{X}$, respectively. 
Let $\iota\colon C(\bar{X}) \to C(X)$ be a
homomorphism defined by $\iota(f)=f|_X$ where $f|_X$ denotes the
restriction of $f$ on $X$.  We will show that in fact the image of
$\iota$ lies in $C_g(X)$.

%We start with the following lemma.
%\begin{lemma}
%\label{lem:x-interpolate}
% For all $R>0$, there exists $S>0$ such that for all
% $v,w\in X$ with $(v\mid w)>S$, 
% there exists $x\in \partial_O X$ satisfying 
% $\min\{(x\mid v),(x\mid w)\}\geq R$.
%\end{lemma}

%\begin{proof}
% Suppose contrarily that there exists $R>0$ such that for all $S>0$,
% there exist $v_S,w_S\in X$ with $(v_S\mid w_S)>S$ satisfying
% $\max\{(x\mid v_S),(x\mid w_S)\}<R$ for all $x\in \partial_O X$.  Since
% $\bar{X}$ is compact, by replacing subsequence, we can assume that the
% sequences $\{v_S\}$ converges to 
%%
% $v_\infty\in \partial_O X$ as $S\to \infty$.  
%%
% Then for sufficiently large $S$, we have 
%%
%\begin{align*}
% (v_\infty \mid v_S) \geq D_2D_3 R \quad \text{ and }\quad
% (v_S\mid w_S)\geq D_2D_3 R.
%\end{align*}
%  \min\{(v_\infty \mid v_S), \, (v_S\mid w_S)\}\geq (D_2D_3)^2 R.
%%
%By Corollary~\ref{cor:qrayultm}, it follows that
% \begin{align*}
%  (v_\infty \mid w_S) \geq (D_2D_3)^{-1}
%   \min\{(v_\infty\mid v_S),\, (v_S\mid w_S)\} \geq R.
% \end{align*}
% This contradicts the assumption.
%\end{proof}

\begin{proposition}
\label{prop:Gromov-fct}
 For all $f\in C(\bar{X})$, the restriction $f|_X$ is a Gromov function.
\end{proposition}

\begin{proof}
 Let $f\in C(\bar{X})$ be a continuous function on $\bar{X}$. Let
 $\{V_n\}_{n\in \N}$ be the fundamental system of entourages of the
 uniform structure on $\bar{X}$ defined in Section~\ref{sec:topon-X-dX}.
 Since $\bar{X}$ is compact, for $\epsilon>0, $ there exists $n$ such that
 for $(x,y)\in V_n$, we have $\abs{f(x)-f(y)}< \epsilon$.

 Now for $v,w\in X$ with $(v\mid w)>n$, we have $(v,w)\in V_n$. Thus 
 $\abs{f(v)-f(w)}< \epsilon$.
 It follows that the restriction $f|_X$ is a Gromov function.
\end{proof}

%By Proposition~\ref{prop:Gromov-fct}, 
Now we have shown that the restriction map
$\iota\colon C(\bar{X})\to C_g(X)$ is well-defined. 
To show the subjectivity of $\iota$, we need the following lemma.

\begin{lemma}
\label{lem:on-q-ray}
 Set $\delta_1:=\lambda(\tilde{\theta}(0))+k_1$.
 Let $\gamma \in \calL_O^\infty$ be a quasi-geodesic ray.
 For $t\in \R_{\geq0}$ and $\gamma_t\in \calL_O$ with domain $[0,a_t]$
 such that $\gamma_t(a_t)=\gamma(t)$,
 we have $(\gamma\mid \gamma_t)> 
 (t-\tilde{\theta}(0))/E\delta_1$.
\end{lemma}

\begin{proof}
 Let $\gamma \in \calL_O^\infty$ be a quasi-geodesic ray.
 For $t\in \R_{\geq0}$, we choose
 $\gamma_t\in \calL_O$ whose domain is $[0,a_t]$,
 such that  $\gamma_t(a_t)=\gamma(t)$. We remark that 
 $\abs{a_t-t}\leq \tilde{\theta}(0)$.
 Since
 \begin{align*}
  \ds{\gamma(a_t),\gamma_t(a_t)}=\ds{\gamma(a_t),\gamma(t)}\leq
  \lambda(\tilde{\theta}(0))+k_1 = \delta_1,
 \end{align*}
 we have
 \begin{align*}
  \ds{\gamma\left(\frac{a_t}{E\delta_1}\right),
  \gamma_t\left(\frac{a_t}{E\delta_1}\right)}
  \leq D+1.
 \end{align*}
 Thus we have 
 $(\gamma\mid\gamma_t)\geq (t-\tilde{\theta(0)})/E\delta_1$.
\end{proof}

%\begin{remark}
% The above argument become slightly simpler if we assume that $\calL$ is
% prefix closed.
%\end{remark}

%\begin{lemma}
% Let $f$ be a Gromov function. Then $f$ is bounded.
%\end{lemma}

%\begin{proof}
% Let $f$ be a Gromov function. There exists $R>0$ such that for all $v,w\in X$ 
% with $(v\mid w)\geq R$, we have $\abs{f(v)-f(w)}< 1$. 
% Set $K_R:=\{v\in X: \ds{v,O}\leq \lambda (E\delta_1 R+\tilde{\theta}(0))+k\}$. 
% For $v\in X\setminus K_R$, we choose $\gamma_v\in \calL_O$ with domain 
% $[0,a_v]$ and $\gamma_v(a_v)=v$. Since $v\in K_R$, 
% we have $a_v\geq E\delta_1R + \tilde{\theta}(0)$. By Lemma~\ref{lem:on-q-ray},
% we have $(v\mid\gamma(E\delta_1R + \tilde{\theta}(0)))\geq R$. 
% Since $\gamma(E\delta_1R + \tilde{\theta}(0))\in K_R$, we have
% $\abs{f(v)}\leq \sup\{f(x):x\in K_R\} + 1$.
%\end{proof}

Now we show that the map $\iota\colon C(\bar{X})\to C_g(X)$ is surjective.
Indeed, we show that every $f\in C_g(X)$ can be extended to $\bar{X}$.

\begin{lemma}
\label{lem:exists}
 Let $f\colon X\to \C$ be a continuous Gromov function. Let $\gamma\in
 \calL_O^\infty$ be a quasi-geodesic ray. Then the limit %$\lim_{n\to \infty}f(\gamma(n))$
 \begin{align*}
  \lim_{n\to \infty}f(\gamma(n))
 \end{align*}
 exists.
\end{lemma}

\begin{proof}
 We will show that the sequence $\{f(\gamma(n))\}_{n\in \N}$ is a Cauchy
 sequence. 
 For $\epsilon> 0$, there exists $N>0$ such that for $v,w\in X$ with 
 $(v\mid w)\geq (D_2D_3)^{-1}(N-\tilde{\theta}(0))/E\delta_1$,
 we have 
 \begin{align*}
  \abs{f(v)-f(w)}< \epsilon.
 \end{align*}
 Here $\delta_1$ is a constant defined in Lemma~\ref{lem:on-q-ray}. 
 Then for $m>n>N$, by Corollary~\ref{cor:qrayultm} and
 Lemma~\ref{lem:on-q-ray}, we have
 \begin{align*}
  (\gamma(n)\mid \gamma(m))&\geq  
  (D_2D_3)^{-1}\min\{(\gamma(n)\mid [\gamma]),\,([\gamma]\mid \gamma(m))\}\geq 
(D_2D_3)^{-1}(N-\tilde{\theta}(0))/E\delta_1.
 \end{align*}
 Thus we have $\abs{f(\gamma(n))-f(\gamma(m))}<\epsilon$. 
 This complete the proof.
 %This shows that
 %$\{f(\gamma(n))\}_{n\in \N}$ is a Cauchy sequence.
\end{proof}

\begin{lemma}
\label{lem:uniqueness}
 For $\gamma,\eta\in \calL_O^\infty$, 
 if $\gamma\sim \eta$, then $\lim_{n\to\infty}(\gamma(n)\mid\eta(n)) = \infty$.
\end{lemma}

\begin{proof}
 Let 
 $\gamma,\eta\in \calL_O^\infty$ be quasi-geodesic rays 
 with $\gamma\sim \eta$.
 By Corollary~\ref{cor:qrayultm} and Lemma~\ref{lem:on-q-ray}, we have
 \begin{align*}
  (\gamma(n)\mid \eta(n))&\geq (D_2D_3)^{-2}
  \min\{
   ([\gamma]\mid\gamma(n)),\, ([\gamma]\mid[\eta]),\, ([\eta]\mid\eta(n))
  \}\\
  &\geq \frac{n-\tilde{\theta}(0)}{(D_2D_3)^{2}E\delta_1} \to \infty.
 \end{align*}
\end{proof}

\begin{corollary}
\label{cor:unique} Let $f\colon X\to \C$ be a Gromov function. For
 $\gamma,\eta\in \calL_O^\infty$, if $\gamma\sim \eta$, 
then $\lim_{n\to \infty}f(\gamma(n)) = \lim_{n\to \infty}f(\eta(n))$.
\end{corollary}

Let $f\colon X\to \C$ be a continuous Gromov function. We extend $f$ to 
a map $\bar{f}\colon \bar{X}\to \C$ by 
$\bar{f}(x):=\lim_{n\to\infty}f(\gamma(n))$ where $x\in \partial_O X$ and
$\gamma\in \calL_O^\infty$ is a representative of $x$.
By Lemma~\ref{lem:exists} and Corollary~\ref{cor:unique}, this extension is
well-defined.

\begin{lemma}
\label{lem:conti-at-infty}
 The above extension $\bar{f}\colon \bar{X}\to \C$ is continuous.
\end{lemma}

\begin{proof}
 We show that for each $x\in \partial_O X$, the map
 $\bar{f}$ is continuous at $x$. We choose $\gamma\in \calL_O^\infty$ which
 is a representative of $x$. For $\epsilon>0$, there exists $T>0$ such
 that for $t\geq T$
 \begin{align*}
  \abs{\bar{f}(x)-f(\gamma(t))}< \frac{\epsilon}{3}.
 \end{align*}
 Since $f$ is a
 Gromov function, there exists $R>0$ such that for $v,w\in X$ with
 $(v\mid w)\geq (D_2D_3)^{-2}R$, we have 
 \begin{align*}
  \abs{f(v)-f(w)}<\frac{\epsilon}{3}
 \end{align*}
 Set $T':= \max\{T, RE\delta_1+\tilde{\theta}(0)\}$.
 By Lemma~\ref{lem:on-q-ray}, we have $(x\mid \gamma(T'))> R$. 

 First let $v\in X$ be a point with $(x\mid v)>R$. 
 It follows that
 \begin{align*}
  (\gamma(T')\mid v) \geq 
  (D_2D_3)^{-1}\min\{(\gamma(T')\mid x),\, (x\mid v)\}
  >(D_2D_3)^{-1}R. 
 \end{align*}
 Therefore we have 
$\abs{\bar{f}(x)-f(v)}\leq \abs{\bar{f}(x)-f(\gamma(T'))}
 +\abs{f(\gamma(T'))-f(v)}<\epsilon$.

 Next let $y\in \partial_O X$ be a point with $(x\mid y)>R$.
 We choose $\eta\in\calL_O^\infty$ which is a representative of $y$.
 There exists $S>0$ such that for $s\geq S$
 \begin{align*}
  \abs{\bar{f}(y)-f(\eta(s))}< \frac{\epsilon}{3}.
 \end{align*}
 Set $S':= \max\{S, RE\delta_1+\tilde{\theta}(0)\}$.
 By Lemma~\ref{lem:on-q-ray}, we have $(y\mid \eta(S'))> R$. 
 This implies 
  \begin{align*}
  (\gamma(T')\mid \eta(S')) \geq 
  (D_2D_3)^{-2}
  \min\{(\gamma(T')\mid x),\, (x\mid y),
   \,(y\mid \eta(S'))\}> (D_2D_3)^{-2}R. 
 \end{align*}
 Therefore we have 
 \begin{align*}
  \abs{\bar{f}(x)-\bar{f}(y)}
  \leq \abs{\bar{f}(x)-f(\gamma(T'))} + \abs{f(\gamma(T'))-f(\eta(S'))}
  +\abs{f(\gamma(S'))-\bar{f}(y)}<\epsilon .
 \end{align*} 
 It follows that $\bar{f}$ is continuous at $x$.
\end{proof}
It follows from Lemma~\ref{lem:conti-at-infty}
that the map $\iota\colon C(\bar{X})\to C_g(X)$ is surjective.
Especially, all function $f\in C_g(X)$ is bounded. 
Moreover, we have the following.
%%%%%% closed under multiplication %%%%%%%%%%%%%%%%%%%%%%%%%%%%%%%
%%
%\begin{proof}
% It is easy to show that $C_g(X)$ is closed under point wise sum.
% For $f,g\in C_g(X)$ and for any $\epsilon>0$, there exists $R>0$ such that
% for $v,w\in X$ with $(v\mid w)\geq R$, we have 
% $\max\{\abs{f(v)-f(w)},\abs{g(v)-g(w)}\}<\epsilon$. Then we have
% \begin{align*}
%  \abs{f(v)g(v)-f(w)g(w)}\leq \abs{f(v)-f(w)}\abs{g(v)}+
%  \abs{f(v)}\abs{g(v)-g(w)}.
% \end{align*}
% It follows that the mulitiplication $fg$ is a Gromov function.
%\end{proof}

\begin{theorem}
 \label{thm:main} Let $C_b(X)$ denote the $C^*$-algebra of bounded
continuous complex valued functions on $X$. The set $C_g(X)$ is a closed
$*$-sub-algebra of $C_b(X)$ and the restriction map $\iota\colon
C(\bar{X})\to C_g(X)$ is an isomorphism.
\end{theorem}

Engel and Wulff introduced the {\itshape combing compactifications} for
proper combing spaces~\cite{combable-corona}. They also showed that a
proper coarsely convex space $X$ admits a proper combing. In fact they
showed that this combing satisfies better condition, {\itshape coherent}
and {\itshape expanding} \cite[Lemma 3.26]{combable-corona}. We denote by 
$\overline{X}{}^H$ the combing compactification of $X$.

\begin{corollary}
\label{cor:idealbdry-eq-combing}
 The identity map on $X$ extends to a homeomorphism 
 $\bar{X}\to \overline{X}{}^H$
 from the ideal boundary compactification to the 
 combing compactification.
\end{corollary}
\begin{proof}
 Engel and Wulff showed the similar statement for proper geodesic Gromov
 hyperbolic spaces and for 
 Gromov boundaries~\cite[Lemma 3.23]{combable-corona}. The key ingredient
 is a functional analytic characterization of the Gromov boundary by Roe
 \cite[Proposition 2.1]{ROEHYPERBOLIC}. Now by Theorem~\ref{thm:main},
 we can apply the argument in the proof 
 of~\cite[Lemma 3.23]{combable-corona} just replacing the Gromov product
 in usual sense by the one in the setting of coarsely convex spaces
 defined in Section~\ref{sec:gromov-product}.
\end{proof}

%\subsection{Topological dimension}

Engel and Wulff obtained many results on groups equipped with
expanding and coherent combings~\cite{combable-corona}.  Here we apply
one of them to groups acting on coarsely convex spaces.

Let $G$ be a group acting geometrically on a
proper coarsely convex space $X$.  Then $G$ is finitely generated, and
$G$ is equipped with a word metric which is quasi-isometric to $X$.
By Proposition~\ref{prop:ceq-pres-qconvity}
the metric space $G$ is coarsely convex. Let $\partial G$ be the ideal
boundary of $G$. We denote by $\dim(\partial G)$ the topological
dimension of $\partial G$. We also denote by $\mathrm{cd}(G)$ the
cohomological dimension of the group $G$.
Combining \cite[Corollary 7.13]{combable-corona} with
Corollary~\ref{cor:idealbdry-eq-combing}, we obtain the following.
\begin{corollary}
\label{cor:cohom-dim}
 Let $G$ be a group acting geometrically on a proper coarsely convex space.
 If $G$ admits a finite model for the classifying space $BG$, then
 \begin{align*}
  \mathrm{cd}(G) = \dim(\partial G)+1.
 \end{align*}
\end{corollary}

%%%%%%%%% Open questions %%%%%%%%%%%%%%%%%%%
%\subsection{Open question}

%\begin{remark}
% It seems natural to ask when a group $G$ acting on a coarsely convex
% space admits finite $G$-simplicial complex which is a universal space
% for proper actions. 
%\end{remark}

%%%%%%%%%%%%%%%%%%%%%%%%%%%%%%%%%%%%%%%%%%%%%%%%%%%%%%%

%%%%%%%%%%%% References %%%%%%%%%%%%%
%%

\bibliographystyle{amsplain} \bibliography{/Users/tomo/Library/tex/math}

\bigskip
%%%%%%%%%%%% Authors' addresses %%%%%%%%%%%%%
\address{ Tomohiro Fukaya \endgraf
Department of Mathematics and Information Sciences,
Tokyo Metropolitan University,
Minami-osawa Hachioji, Tokyo, 192-0397, Japan
}

\textit{E-mail address}: \texttt{tmhr@tmu.ac.jp}

\address{ Shin-ichi Oguni\endgraf
Department of Mathematics, Faculty of Science,
Ehime University,
2-5 Bunkyo-cho,
Matsuyama,
Ehime,
790-8577, Japan
}

\textit{E-mail address}: \texttt{oguni@math.sci.ehime-u.ac.jp}

%%%%%%%%%%%%%%%%%%%%%%%%%%%%%%%%%%%%%%%%%%%%%%%%%%%%%%%%%%%%%%%%%%%%%%%
%%
%%   Revision history

%\newpage

\end{document}